\documentclass[a4paper,12pt,reqno]{amsart}
\usepackage{hyperref}
\hypersetup{backref,colorlinks=true,allcolors=black}
\usepackage{mathrsfs}
\usepackage{enumitem}
\usepackage{dsfont}
\usepackage{mathtools}
\usepackage{xcolor}
\usepackage[final]{graphicx}
\usepackage{upgreek}
\usepackage{amsmath,amssymb}
\usepackage{nameref}

\newtheorem{theorem}{Theorem}[section]
\newtheorem{proposition}[theorem]{Proposition}
\newtheorem{lemma}[theorem]{Lemma}
\newtheorem{corollary}[theorem]{Corollary}

\theoremstyle{definition}

\newtheorem{example}[theorem]{Example}
\newtheorem{remark}[theorem]{Remark}

\newtheorem{innercustomthm}{Assumption}
\newenvironment{customassumption}[1]
{\renewcommand\theinnercustomthm{#1}\innercustomthm}
{\endinnercustomthm}

\newtheorem{innercustomgeneric}{\customgenericname}
\providecommand{\customgenericname}{}
\newcommand{\newcustomtheorem}[2]{%
  \newenvironment{#1}[1]
  {%
   \renewcommand\customgenericname{#2}%
   \renewcommand\theinnercustomgeneric{##1}%
   \innercustomgeneric
  }
  {\endinnercustomgeneric}
}

\newcustomtheorem{customthm}{Theorem}
\newcustomtheorem{customlemma}{Lemma}
\newcustomtheorem{customassump}{Assumption}
\newcustomtheorem{customproof}{Proof}

\newcommand{\ignore}[1]{}
\newcommand{\R}{\mathbb{R}}
\newcommand{\N}{\mathbb{N}}

\newcommand{\on}[1]{\operatorname{#1}}
\newcommand{\norm}[1]{\left\lVert #1 \right\rVert}
\newcommand{\abs}[1]{\left\vert #1 \right\rvert}
\newcommand{\E}[1]{\mathbb{E}{\left[ #1\right]}}

\newcommand{\range}{\mathrm{Range}}

\newcommand{\ceil}[1]{\lceil #1 \rceil}

\DeclarePairedDelimiter\autobracket{(}{)}
\newcommand{\brac}[1]{\autobracket*{#1}}
\newcommand{\inner}[1]{\left\langle #1 \right\rangle}

\usepackage{microtype}

\theoremstyle{definition}

\makeatletter
\@addtoreset{proofcase}{theorem}
\makeatother

\theoremstyle{definition}

\makeatletter
\@addtoreset{proofstep}{theorem}
\makeatother

\usepackage{microtype}
\usepackage{parskip}

\definecolor{crimson}{rgb}{0.7294,0.0666,0.0470}
\begin{document}
\bibliographystyle{amsalpha}

\author{Solesne Bourguin}\thanks{Solesne Bourguin was partially supported by the Simons Foundation Award 635136.}
\address{$^1$Department of Mathematics and Statistics, Boston University, 665 Commonwealth Avenue, Boston, MA 02215, USA}
\email{bourguin@math.bu.edu}
\author{Thanh Dang}
\address{$^2$Department of Computer Science, University of Rochester, 2513 Wegmans Hall, Rochester, NY 14627, USA}
\email{ycloud777@gmail.com}
\author{Yaozhong Hu}\thanks{Yaozhong Hu was supported by an NSERC discovery fund and centennial fund of University of Alberta.} 
\address{$^3$Department of Mathematical and Statistical Sciences, University of Alberta, Edmonton, AB T6G 2G1, Canada}
\email{yaozhong@ualberta.ca}
\title[]{Non-central limit of densities of some functionals of Gaussian processes}

\begin{abstract}
We establish the convergence of the densities of a sequence
of nonlinear functionals of an underlying Gaussian
process  to the density of a Gamma distribution. The key idea of our work is a new density formula for random variables in the setting of Markov diffusion generators, which yields a special representation for the density of a Gamma distribution. Via this representation, we are able to provide precise
estimates on the distance between densities while developing the techniques of
Malliavin calculus and Stein's method suitable to Gamma approximation at the density level. We
first focus our study on the case of random variables living in a
fixed Wiener chaos of an even order for which the   bound for the difference of the densities  
can be dominated by  a linear combination of moments up to
order four. We then study the case of general Gaussian functionals
with possibly infinite chaos expansion. Finally, we provide an
application to random variables
living in the second Wiener chaos. 
\end{abstract}
\subjclass[2010]{60F05, 60H07}
\keywords{convergence of densities; Markov diffusion generator; Malliavin calculus; Stein's
  method; Multiple Wiener-It\^o integrals; Wiener chaos.}
\maketitle



\section{Introduction}
Central limit theorems  of (nonlinear) functionals of a Gaussian
process, which we shall refer to as Gaussian functionals, have been widely studied in the past decades, in particular
by making use of Malliavin calculus (given the
Gaussian underlying context). Central asymptotic results are now also
available in a quantitative way, i.e. via convergence rates  with respect to 
various  metrics/probability distances. In particular, the seminal work \cite{NP09main} combines Stein's method and Malliavin calculus to obtain quantitative central limit theorems on Gaussian spaces, pioneering what is known today as the Malliavin-Stein method for studying quantitative approximation problems. Among the results obtained in \cite{NP09main} is the four moment theorem, which asserts that for a sequence of multiple Wiener integrals $F_n$ of   fixed order, convergence in laws to a normal distribution $N$ is equivalent to convergence of the first four moments of $F_n$ to the first four moments of $N$. 

The reason why multiple Wiener integrals are important objects to study is because any square-integrable nonlinear functionals of Gaussian processes can be decomposed
into a possibly infinite sum of multiple Wiener integrals of
increasing orders (known as the Wiener chaos decomposition), making them the building blocks of Gaussian functionals. Hence, one way to go about studying
Gaussian functionals is to study the behavior of these  building
blocks. This is the route that was taken by the above-mentioned
reference and that has proven to be extremely successful and fruitful
for the last fifteen years in building a research communities around the Malliavin-Stein method. 

Motivated by applications to statistics
  	and  others fields, the question of sufficient conditions for the  convergence of the densities  of 
  	such  nonlinear Gaussian functionals to a normal density  was posed and
  	answered in \cite{HLN14}, where uniform bounds for the
  	convergence of densities and their derivatives were obtained via the Malliavin-Stein method. One of the results in \cite{HLN14} is a four moment theorem for density convergence of a sequence of multiple Wiener integrals $F_n$ to a normal density $N$ (compared to the one in \cite{NP09main} about convergence in laws). Informally, the result says if the first four moments of $F_n$ converge to the corresponding first four moments of $N$, and additionally if negative moments of certain order of the Malliavin norm of $F_n$ exist, then densities of $F_n$ and their derivatives up to a certain order will converge in the uniform metric to the normal density and its derivatives. On a related note, the authors of a recent paper \cite{herry2024superconvergence}  observe what they call the superconvergence phenomenon: if a sequence of multiple integrals $F_n$ converges in distribution to a normal density $N$, then densities of $F_n$ and their derivatives (of all orders) respectively converge in the uniform metric to the normal density and its derivatives. This is a significant improvement of the four moment theorem in \cite{HLN14}. We have some technical comments about \cite{herry2024superconvergence} in a paragraph below.


In recent years, with the emergence of new tools and techniques, there has been a growing interest in pushing the
asymptotic study of functionals of Gaussian processes beyond the
central regime. For instance, the reference \cite{ACP14} studies under which conditions the limiting distributions
are the Gamma and Beta ones. Later on, in \cite{bourguintaqqu2019}, the whole
Pearson class of distribution is considered and quantitative estimates
for convergence in distribution to elements of this class are
obtained. 

The goal of this paper is to go further to investigate the
convergence of densities (and their derivatives) of functionals of Gaussian processes in a
non-central regime, i.e. when the limiting law is not Gaussian but a Gamma distribution. To accomplish this goal, we derive a new density formula for random variables in the setting of Markov diffusion generators (see Proposition \ref{prop_densityrepgeneral}), which then leads to a representation for the density of a Gamma distribution. {In fact, the result in  Proposition \ref{prop_densityrepgeneral} along with Proposition \ref{prop_densitygeneralformderivatives} can be applied to get density representations for many other targets that are invariant probability measures of Markov processes (see Remark \ref{remark_othertargets}).} Moreover, we are able to  extend the Malliavin-Stein method to be suitable to this new situation of approximation of a Gamma density. 

It is worth pointing out that in the Malliavin-Stein literature, the majority of references are concerned with  convergence in some probability distance and/or convergence in laws. Only a few references such as \cite{herry2024superconvergence,HLN14, hutindel2015density,kuzgun2022convergence} investigate density convergence of random variables on Wiener chaos, which however   focus on the central regime, i.e. density convergence to a normal one. {Thus, there remains much to investigate regarding the Malliavin-Stein method at the density level.}
   
The two main results of this paper address the case where the Gaussian
functional is a multiple Wiener integral and the case where the
Gaussian functional has a possibly infinite chaos expansion,
respectively. For the first case, we provides pointwise
estimates between the density of a random variable $F$ living in a
Wiener chaos of an even order and the density of a Gamma random variable. As an extension to this, we also provide pointwise estimates for the derivatives of the densities. In particular, our main estimates take a
particularly simple form: they are given as linear combination of the first four moments of $F$ and the first four moments of the Gamma random variables. For the second case,
whenever $F$ has a possible infinite chaos expansion, we obtain density estimates based on operators in Malliavin
calculus, keeping in mind that bounds based on moments cannot be obtained in this
context. 

To state our main results we will denote by
$\mathcal{G}_{\alpha}$ the Gamma distribution of parameter $\alpha>0$ with density function given by
\begin{align}
	\label{formula_gammadensity_original}
	p_{\mathcal{G}_\alpha}(x) =\begin{cases} 
		\frac{1}{\Gamma(\alpha)}x^{\alpha-1}e^{-x} &  x>0 \\
		0 & x \leq 0
	\end{cases}
\end{align}
for which it is easily seen by induction that for $k\geq 1$,
\begin{align}
	\label{formula_gammadensityderivative_original}
	p^{(k)}_{\mathcal{G}_\alpha}(x) =\begin{cases} 
		\displaystyle \frac{(-1)^k}{\Gamma(\alpha)}x^{\alpha-1}e^{-x}\sum_{i=0}^{k}{k\choose i}\prod_{j=1}^i(j-\alpha)\frac{1}{x^i}&  x>0 \\
		0 & x\leq 0\,, 
	\end{cases},
\end{align}
where $p^{(k)}_{\mathcal{G}_\alpha}(x)$ denotes the $k$-th derivative
of $p_{\mathcal{G}_\alpha}(x)$. Moreover, $p^{(k)}_{\mathcal{G}_\alpha}(0)=0$ when $\alpha>k+1$ and undefined otherwise. Here we adopt the convention $\prod_{k=1}^0\alpha_k=1$ for any $\alpha_k$.

In what follows, we denote by $D$ the Malliavin derivative operator
acting on the isonormal Gaussian process underlying the functional
$F$. This isonormal Gaussian process is assumed to be indexed by a
Hilbert space $\mathfrak{H}$. All these notions and objects will be
rigorously defined in Subsection \ref{malliavinsubsection}, to which we
refer the reader for more details.

In the case where $F$ is a multiple Wiener integrals, our main
findings can be split in two results which deal with the densities
themselves and their derivatives, respectively. The first one is
stated below.
\begin{theorem}
\label{theorem_fourthmoment_singlechaos_0thderivative}
Let $F$ be a multiple Wiener integral of an even order $q\geq 2$ such
that $\E{F^2}=\alpha$. Denote by $p_{F+\alpha}$ the density function
of $F +\alpha$. Then, for every $x \neq 0$ and $\alpha >0$, one has
the pointwise density estimate
\begin{align*}
     \abs{p_{F+\alpha}(x)-p_{\mathcal{G}_\alpha}(x)} &\leq P_0(x)\sqrt{\frac{q^2}{3}\brac{\E{F^4}-6\E{F^3}+6(1-\alpha)\alpha+3\alpha^2}}\, ,
\end{align*}
where $P_0(x)$ is a positive quantity depending on $x$ given by 
\begin{align*}
&d_1(x)\sum_{i=1}^{\ceil{\alpha}-1}\E{(F+\alpha)^{2(\alpha-i)}}^{1/2}+
  {d_2(x)}\E{(F+\alpha)^{-2}}^{1/2}+{d_3(x)}\\
&\quad +\E{\norm{DF}_{\mathfrak{H}}^{-4}}^{1/2}+\abs{\alpha-1}\E{\norm{DF}_{\mathfrak{H}}^{-4}(F+\alpha)^{-2}}^{1/2} +C_1\E{\norm{DF}_{\mathfrak{H}}^{-6}}^{1/2},
\end{align*}
where the factors $d_1(x),d_2(x),d_3(x)$ are positive and finite for every $x$, and $C_1$ is some positive constant independent of $x$. Moreover, for $x=0$ the same estimate holds if $\alpha>1$. 
\end{theorem}
We emphasize that the above density estimate is tight, in the sense that if $F+\alpha$ is identically distributed to $\mathcal{G}_\alpha$ then the right hand side of the above estimate equals zero. We refer to Remark \ref{remark_tightestimate} for  further explanation.

As far as the derivatives of the above densities are concerned, we can
state the following estimates, which naturally uses higher order  Malliavin derivatives. 
\begin{theorem}
\label{theorem_fourthmoment_singlechaos_kthderivative}
Let $F$ be a multiple Wiener integral of an even order $q\geq 2$ such
that $\E{F^2}=\alpha$ and $1/\norm{DF}^2_\frak{H}\in \cap_{p\geq 1}L^p$. Denote by $p_{F+\alpha}$ the density function
of $F +\alpha$. Let $k\geq 1$ be any integer. Then, for every $x \neq
0$ and $\alpha >0$, one has the pointwise density estimate
\begin{align*}
     &\abs{p^{(k)}_{F+\alpha}(x)-p^{(k)}_{\mathcal{G}_\alpha}(x)}\leq P_k(x)\sqrt{\frac{q^2}{3}\brac{\E{F^4}-6\E{F^3}+6(1-\alpha)\alpha+3\alpha^2}}\, ,
\end{align*}
where $P_k(x)$ is a quantity depending on $x$, as well as on
$\E{\norm{DF}_{\mathfrak{H}}^{-12i}}$, $\E{(F+\alpha)^{-6j}}$ and $\E{(F+\alpha)^{2(\alpha-n)}}$ for $1\leq i\leq 2k+2$, $1\leq j\leq k+1$ and $1\leq n\leq \ceil{\alpha}-1$. Moreover for $x=0$, the same estimate holds if $\alpha>k+1$.     
\end{theorem}

In the second case, namely when $F$ is a general random variable, in
the sense that it has a possibly infinite Wiener chaos expansion, our main
result is an estimate on the distance  of the density of $F$ and the
Gamma density. The estimate here is not as simple or explicit as in the case of   random variables living in a single Wiener chaos since it cannot be expressed as a linear combination of moments. This estimate
is the object of Theorem \ref{theorem_gammaconvergence_sumofchaos},
which we do not restate here since it involves  some additional  notations.

    It is interesting to point out that in Theorems \ref{theorem_fourthmoment_singlechaos_0thderivative} and   \ref{theorem_fourthmoment_singlechaos_kthderivative} it is  assumed that  the order $q$ of the multiple Wiener integrals is always an even integer greater or equal to $2$. The reason for this  is that  in these two theorems  we aim to control the density difference by fourth moment    bounds and in order to complete this goal, we use the fact that the order $q$ is even to    obtain  two important technical estimates \eqref{estimate_D^2F-qDF} and \eqref{estimate_DkDlbymalliavindifference}. We are not sure  if these technical estimates are still true when $q$ is an odd integer.  Another  reason is a well-known fact: any multiple integral of an odd order has a third moment equal to $0$, while we know $\E{\mathcal{G}_\alpha^3}=\alpha(\alpha+1)(\alpha+2)$. This means that  if $\{F_n:n\in\N\}$ is a sequence of multiple Wiener integrals of some fixed odd order, then $F_n$ (which is centered by definition) cannot converge to the centered Gamma random variable $\mathcal{G}_\alpha-\alpha$. On the other hand, if one 
    is  interested in controlling   differences of densities by using
    the Malliavin derivatives, then one can use Theorem \ref{theorem_gammaconvergence_sumofchaos} instead. 

    In view of the substantial improvement in \cite{herry2024superconvergence} to the four moment theorem for density convergence to a normal target in \cite{HLN14}, a reasonable question is whether the techniques in \cite{herry2024superconvergence} can be replicated in the context of density convergence to a Gamma target. We do not have a definite answer to the question, however we provide a few observations here. A significant portion of \cite{herry2024superconvergence} is devoted to bounding the negative moments of $\norm{DF_n}^2_{\frak{H}}$ where $F_n$ belongs to a sequence of multiple Wiener integrals of a fixed order, which then allows the authors of \cite{herry2024superconvergence} to deduce density convergence of $F_n$ to a normal target. We invite readers to look at the aforementioned reference to appreciate their novel use of spectral techniques to bound the negative moments of $\norm{DF_n}^2_{\frak{H}}$. In the case of density convergence of $F_n$ to a gamma target, our Theorem \ref{theorem_fourthmoment_singlechaos_0thderivative} suggests that one might also need to bound the negative moments of $F_n$. In addition, Theorem 9 of \cite{herry2024superconvergence} is about a sequence of sum of multiple Wiener integrals $F_n=\sum_{i=1}^m J_i(F_n)$ (where $m$ is a non-zero positive integer and $J_i$ denotes the projection into the $i$-th Wiener chaos) and assumes normal convergence along the Wiener chaos of highest order: $J_m(F_n)\to N(0,1)$ in laws. Therefore, even if an analog result can be derived for Gamma density convergence, it is not comparable to our Theorem \ref{theorem_gammaconvergence_sumofchaos} which can be applied to random variables that have infinite Wiener chaos decomposition, and  does not require convergence along the Wiener chaos of highest order.



Let us now give a summary of our methodology. To prove  Theorem \ref{theorem_fourthmoment_singlechaos_0thderivative}, 
  we first derive in Proposition \ref{prop_densityrepgeneral} a density formula for random variables in the setting of Markov diffusion generators, which covers more than just random variables on Wiener chaos. Then by specializing our result to the case of the Laguerre generator, we obtain the following formula for the density of $\mathcal{G}_\alpha$. 
\begin{align*}
    p_{\mathcal{G}_\alpha}(x)= \E{\mathds{1}_{\{\mathcal{G}_\alpha>x\} }\brac{1+\frac{1-\alpha}{\mathcal{G}_\alpha}} }. 
\end{align*}
We can use the well-known  integration by parts formula 
\begin{align*}
	p_{F+\alpha}(x)=\E{\mathds{1}_{\{F+\alpha>x \}}\delta\brac{ \frac{DF}{\norm{DF}^2_\frak{H}}  }}
\end{align*}
to represent the density for $F+\alpha$, 
where $\delta$ is the adjoint of the Malliavin derivative $D$. 
Our next goal is to bound the difference of the two terms on the right hand side of 
\begin{align*}
    p_{F+\alpha}(x)-p_{\mathcal{G}_\alpha}(x)= \E{\mathds{1}_{\{F+\alpha>x \}}\delta\brac{ \frac{DF}{\norm{DF}^2_\frak{H}}  }}- \E{\mathds{1}_{\{\mathcal{G}_\alpha>x\} }\brac{1+\frac{1-\alpha}{\mathcal{G}_\alpha}} }\,.  
\end{align*}
However,  these two terms appear to be completely different, it is natural to wonder
how they can be close each other. To control effectively this difference 
we decompose the first term on the above  right hand side as 
\begin{align*}
    \E{\mathds{1}_{\{F>x \}}\delta\brac{ \frac{DF}{\norm{DF}^2_\frak{H}}  }}= \E{\mathds{1}_{\{F+\alpha>x\} }\brac{1+\frac{1-\alpha}{F+\alpha}} }+\E{\mathds{1}_{\{F+\alpha>x\} }T_1}\,, 
\end{align*}
where the remainder term $T_1$ is specified at \eqref{equation_T1} in Section 3. This implies 
\begin{align}
    \label{step_differencedensity}
    \nonumber&\abs{p_{F+\alpha}(x)-p_{\mathcal{G}_\alpha}(x)}\\
   & \leq \abs{\E{\mathds{1}_{\{F+\alpha>x\} }\brac{1+\frac{1-\alpha}{F+\alpha}} }-  \E{\mathds{1}_{\{\mathcal{G}_\alpha>x\} }\brac{1+\frac{1-\alpha}{\mathcal{G}_\alpha}} }} +\E{\mathds{1}_{\{F+\alpha>x\} }T_1}. 
\end{align}
Now the two terms inside the absolute value sign appear to have similar pattern and   will be bounded by Stein's method,  while the second  term $\E{\mathds{1}_{\{F+\alpha>x\} }T_1}$ will be bounded by using the product formula for multiple Wiener integrals, which involves  the assumption that  the order of the multiple Wiener integral $F$ is an even number in order to obtain the estimate \eqref{estimate_D^2F-qDF}.  An intermediate goal is to    deduce from \eqref{step_differencedensity} the density bound by using the Malliavin derivative: 
\begin{align*}
    \abs{p_{F+\alpha}(x)-p_{\mathcal{G}_\alpha}(x)}\leq P_0(x) \E{\brac{q(F+\alpha)-\norm{DF}^2_\frak{H}}}^{1/2},
\end{align*}
where $P_0(x)$ is the quantity containing negative moments stated in Theorem \ref{theorem_fourthmoment_singlechaos_0thderivative}. As the final step in the proof of Theorem \ref{theorem_fourthmoment_singlechaos_0thderivative}, we bound $\E{\brac{q(F+\alpha)-\norm{DF}^2_\frak{H}}}$ with a combination of the first four moments of $F$ using known results in   \cite{NP09main} and  \cite{ACP14}.

Next, we extend the above result in two other directions. First, for any positive integer $k$ and under appropriate conditions such as   the multiple integral $F$ is $k$-times Malliavin differentiable, we provide a four moment estimate for the $k$-th derivatives of the densities, that is $p_F^{(k)}(x)-p_{\mathcal{G}_\alpha}^{(k)}(x)$ (i.e. Theorem 
	\ref{theorem_fourthmoment_singlechaos_kthderivative}). 
The first key point is to write $ 
	p^{(k)}_{\mathcal{G}_\alpha}(x)=\E{\mathds{1}_{\{\mathcal{G}_\alpha>x\}}\nu_{k+1}(\mathcal{G}_\alpha) }$ 
with  $\nu_{k+1}$ being  given by \eqref{def_nuk} (see Corollary \ref{corollary_densitygamma})
and then one needs to decompose   $p_F^{(k)}(x)$ into terms of similar pattern plus ``higher order" terms.
These ideas  miraculously work out successfully. 
The second extension is  to replace the single chaos random variables by general ones that have  possibly infinite Wiener chaos expansions for which we  bound the difference of   densities by Malliavin derivatives  via the Poincar\'e inequality.

The rest of the paper is organized as follows: Section
\ref{prelimsection} contains preliminaries on the framework of Markov diffusion generators and the basics of Malliavin calculus that we need, while Section \ref{densityformulasection} compiles our
results on representation of densities and their derivatives for the
Gamma distribution and multiple Wiener integrals. Section \ref{sectionsteinmethod}
introduces and develops a version of Stein's method for Gamma
approximations containing a fine study of the solution to the
associated Stein equation. Our version of Stein's method and the
study of the solution to the Stein equation are different from that 
 in \cite{NP09main, DP18} and are of independent
interest.  Section \ref{sectionfixedchaos} contains our main results for random
variables living in a fixed Wiener chaos and Section \ref{section_generalrv} presents
our findings for general random variables with possibly infinite chaos
expansion. Finally, Section \ref{sectionapplications} contains an application of our
 result: a quantitative comparison between the Gamma density and the density of
random variables living in the second Wiener chaos (which is a Wiener
chaos known to contain Gamma distributions). 

\section{Preliminaries}
\label{prelimsection}
\subsection{Markov diffusion generators}
~\
\label{section_Markovdiffusionintro}

In Section \ref{densityformulasection}, we will provide a density
representation for random variables in the setting of Markov diffusion
generators and then deduce a density representation for the gamma
distribution from the previous result. We present here a brief summary
of the framework of Markov diffusion generators based on
\cite{bakry2013analysis}.

The framework consists of some underlying reversible diffusive Markov process $\left\{
  X_t\colon t\geq 0 \right\}$ with
an invariant probability measure $\nu$, associated semigroups $\left\{ P_t\colon t\geq
0\right\}$, infinitesimal
generator $L$ and carré du champ $\Gamma$, where all of these objects
are inherently connected. From an abstract
point of view, a standard and elegant way to introduce this setting is through so
called Markov triples, where one starts from the invariant measure $\nu$, the
carré du champ $\Gamma$ and a suitable algebra of functions (random
variables), from which the generator $L$, the semigroup $\left\{ P_t\colon t\geq
0\right\}$ (including
their $L^2$-domains) and thus also
the Markov process $\left\{ X_t\colon t\geq
0\right\}$ are constructed.

The assumptions we   make on $(E,\nu,\Gamma)$  here are those of the  so-called 
\emph{diffusion properties}
 in the sense of \cite[Part I,
Chapter 3]{bakry2013analysis}, which we summarize below together with some useful consequences of the assumptions. 
\begin{enumerate}[label=\roman*)]
\item $(E,\mathcal{F},\nu)$ is a probability space and $L^2(E,\mathcal{F},\nu)$ is separable.
\item $\mathcal{A}$ is a vector space of real-valued random variables on $(E,\mathcal{F},\nu)$. It is stable under products and under the action of
 smooth functions $\Psi \colon \R^k \to \R$.
\item $\mathcal{A}_0$ as a sub-algebra of $\mathcal{A}$ contains bounded functions which are dense in $L^p(E,\nu)$ for every $p \in [1,\infty)$. We assume
  that $\mathcal{A}_0$ is also stable under the action of smooth functions $\Psi$
  as above and also that $\mathcal{A}_0$ is an ideal in $A$ (if $F\in
  \mathcal{A}_0$ and $G \in \mathcal{A}$, then $FG \in \mathcal{A}_0$). 
\item The carr\'{e} du champ operator $\Gamma \colon \mathcal{A}_0 \times \mathcal{A}_0 \to
  \mathcal{A}_0$ is a bi-linear symmetric map such that $\Gamma(F) \geq 0$ for all $F \in
  \mathcal{A}_0$. For every $F \in \mathcal{A}_0$ there exists a finite constant $c_F$ such
  that for every $G \in \mathcal{A}_{0}$
  \begin{equation*}
 \abs{\E{\Gamma(F,G)}}    \leq c_F \norm{G}_2,
  \end{equation*}
  where $\norm{G}_2^2 = \int_E G^2 d\nu$.
  The Dirichlet form $\mathcal{E}$ is defined on $\mathcal{A}_0 \times \mathcal{A}_0$ by
  \begin{equation*}
    \mathcal{E}(F,G) = \E{\Gamma(F,G)}.
  \end{equation*}
  \item
  The domain $\operatorname{Dom}(\mathcal{E}) \subseteq L^2(E,\nu)$ is obtained by completing
  $\mathcal{A}_0$ with respect to the norm $\norm{F}_{\mathcal{E}} = \left(
    \norm{F}_2 + \mathcal{E}(F,F) \right)^{1/2}$. The Dirichlet form 
  $\mathcal{E}$ and the carr\'{e} du champ operator $\Gamma$ are extended to
  $\operatorname{Dom}(\mathcal{E}) \times \operatorname{Dom}(\mathcal{E})$ by continuity and
  polarization. We thus have that $\Gamma \colon \operatorname{Dom}(\mathcal{E}) \times
  \operatorname{Dom}(\mathcal{E}) \to L^1(E,\nu)$.
\item $L$ is a self-adjoint linear operator, defined on $\mathcal{A}_0$
  via the \emph{integration by parts formula}
  \begin{equation}
    \label{intbyparts_gammaandL}
  \E{GLF}=\E{FLG}  = -\E{\Gamma(F,G)}
  \end{equation}
  for all $F,G \in \mathcal{A}_0$. We assume that $L(\mathcal{A}_0) \subseteq
  \mathcal{A}_0$ and that for every $F\in \mathcal{A}_0$, $ \E{LF}=0$. Moreover, the first part of the above equation is known as the reversible property of $L$ and is related to reversibility
in time of the associated Markov process whenever the initial law is the invariant measure $\nu$.

\item The domain $\operatorname{Dom}(L) \subseteq \operatorname{Dom}(\mathcal{E})$ 
  consists of all $F \in \operatorname{Dom}(\mathcal{E})$ such that
  \begin{equation*}
    \abs{\mathcal{E}(F,G)} \leq c_F \norm{G}_2
  \end{equation*}
  for all $G \in \operatorname{Dom}(\mathcal{E})$, where $c_F$ is a finite constant. The
  operator $L$ is extended from $\mathcal{A}_0$ to $\operatorname{Dom}(L)$ by the
  integration by parts formula~\eqref{intbyparts_gammaandL}. The norm in this domain is defined as 
  \begin{align*}
      \norm{F}_{\operatorname{Dom}(L)}=\brac{\norm{F}_2^2+\norm{LF}_2^2}^{1/2}.
  \end{align*}

  Moreover, we assume that $L 1 = 0$ and
  that $L$ is ergodic: $LF =0$ implies that $F$ is constant for all $F \in \operatorname{Dom}(L)$.
\item The operator $L \colon \mathcal{A} \to \mathcal{A}$ is an extension of $L
  \colon \mathcal{A}_0 \to \mathcal{A}_0$. On $\mathcal{A} \times \mathcal{A}$, the
  carr\'{e} du champ $\Gamma$ is defined by
  \begin{equation*}
    \label{def_gammabyL}
    \Gamma(F,G) := \frac{1}{2} \left( L(FG) - FLG - GLF \right).
  \end{equation*}

\item For all $F \in \mathcal{A}$, we assume $\Gamma(F) \geq 0$ with equality if, and only
  if, $F$ is constant.
\item The \emph{diffusion property} holds. For all $\mathcal{C}^{\infty}$-functions
  $\Psi \colon \R^p \to \R$ and $F_1,\dots,F_p,G \in \mathcal{A}$ one has
  \begin{equation}
    \label{diffusionproperty_gamma}
    \Gamma \left( \Psi \left( F_1,\dots,F_p \right),G \right)
    =
    \sum_{j=1}^p \partial_j \Psi(F_1,\dots,F_p) \, \Gamma(F_j,G)
  \end{equation}
  and
  \begin{equation*}
    \label{eq:30}
    L \Psi(F_1,\dots,F_p) = \sum_{i=1}^p \partial_i \Psi(F_1,\dots,F_p) L F_{i}
    +
    \sum_{i,j=1}^p \partial_{ij} \Psi(F_1,\dots,F_p) \, \Gamma(F_i,F_j).
  \end{equation*}
This diffusion property implies that the semigroup associated with the Dirichlet form $ \mathcal{E}(F, F)$ 
is positive preserving  since 
\[
 \mathcal{E}( \Psi(F), \Psi(F) )
=\mathbb{E}\left[ \Gamma\left( \Psi(F), \Psi(F)
\right)\right] 
=\mathbb{E}\left[ \left(\Psi'(F)\right)^2 \Gamma(F, F)\right] 
 \le \mathcal{E}( F, F )
 \]
 if $|\Psi'|\le 1$ (e.g. 
 \cite[Equation (3.1.12)]{bakry2013analysis}). 
 
\item The integration by parts formula~\eqref{intbyparts_gammaandL} continues to hold if $F \in
  \mathcal{A}$ and $G \in \mathcal{A}_0$ (or vice versa).
  \item The eigen-functions of $L$ in $\operatorname{Dom}(L)$ form the eigen-subspaces of $L$ known as Markov diffusion chaos. 
     
\end{enumerate}

For an $F\in \range(L)$ with $\mathbb{E}(F)=0$,  
we define $L^{-1}F=G$ if $F=LG$.  This is
well-defined since if $LG_1=LG_2$,   or $L(G_1-G_2)=0$,
then $\mathbb{E}(G_2-G_1)=0$ by our ergodic assumption on $L$ 
(part (vii)).   In general, for any $F\in \range(L)$,
we define $L^{-1} F=L^{-1} (F-\mathbb{E}(F))$.  It is easy to verify  that 
  the  inverse $L^{-1}$   satisfies that  for any $F\in \range (L)$  
  \begin{align}
      \label{relation_Landinverse}
      LL^{-1}F=L^{-1}LF=F-\E{F}\,. 
  \end{align}
%


\begin{example}
\label{example_laguerre}
    An example which we will revisit in Corollary \ref{corollary_densitygamma} is the Full Markov triple associated with the one-dimensional Laguerre process (\cite[Section 2.7.3]{bakry2013analysis}): $E=\R$, $\nu$ is a gamma distribution with density $\frac{1}{\Gamma(\alpha)}x^{\alpha-1}e^{-x},\alpha>0$ and $\Gamma(f)(x)=x(f'(x))^2$, with $x$ distributed as $\nu$. The Laguerre generator is given by $Lf(x)=xf''(x)+(\alpha-x)f'(x)$ and a suitable algebra $\mathcal{A}$ is the set of polynomials. Moreover, regarding the Laguerre semigroups, one can find a Mehler-type formula for them in \cite{arras2017stroll}. Finally, the eigen-functions of $L$ are the Laguerre polynomials $\{Q_n:n\geq 1 \}$ defined by
    \begin{align*}
       Q_n(x)=\frac{x^{-\alpha}e^x}{n!}\frac{d^n}{dx^n}\brac{e^{-x} x^{n+\alpha}}. 
    \end{align*}
\end{example}

In the upcoming section, we will give an overview of Malliavin calculus on Gaussian space. In fact, it is another example of a Full Markov Triple. However, it does have additional structure such as the product formula that is not standard for a Full Markov Triple. We will introduce Malliavin calculus in a standard way following \cite{nualart2006malliavin} and only at the end relates it back to the framework of Markov diffusion generators. 
\subsection{Malliavin calculus on Gaussian space}
\label{malliavinsubsection}
~\

Let $\mathfrak{H}$ be a real separable Hilbert space with orthonormal basis $\{e_i:i\in \N \}$ and $\left\{ Z(h)
\colon h \in \mathfrak{H} \right\}$ an {isonormal Gaussian process} defined on some probability space $(\Omega,\mathcal{F},P)$ and
indexed on $\mathfrak{H}$, that is, a centered Gaussian family of random variables 
such that 
\begin{align*}\E{Z(h)Z(g)} = \left\langle h,g \right\rangle_{\mathfrak{H}}.\end{align*} 
The associated probability space $(\Omega,\sigma(Z),P)$ is called a {Gaussian space}. The {Wiener chaos} of order $n$, denoted by $\mathcal{W}_n$, is the closed linear span
of the random variables $\left\{ H_n(Z(h)) \colon h \in \mathfrak{H},\
\norm{h}_{\mathfrak{H}}=1 \right\}$ where $H_{n}$ is the $n$-th Hermite
polynomial given by $H_{0}=1$ and for $n\geq 1$,
\begin{equation*}
	H_{n}(x)=\frac{(-1)^{n}}{n!} \exp \left( \frac{x^{2}}{2} \right)
	\frac{d^{n}}{dx^{n}}\left( \exp \left( -\frac{x^{2}}{2}\right)
	\right), \quad x\in \mathbb{R}.
\end{equation*} 
Wiener chaos are the building blocks of the Gaussian space as we have the orthogonal decomposition
\begin{align}
	\label{decomposition_hermite}
	L^2(\Omega,\sigma(Z),P)= \bigoplus_{n=0}^\infty \mathcal{W}_n. 
\end{align}
Next, we will introduce {multiple Wiener integrals} with respect to the Gaussian process $Z$. In order to do that, let us first explain what it means by symmetrization of tensor products. For an integer $q\geq 2$, the expansion of $f\in\mathfrak{H}^q$ is given by \\$\sum_{i_1,\ldots,i_n=1}^\infty a(i_1,\ldots,i_n)e_{i_1}\ldots e_{i_n}$. Then the symmetrization of $f$, denoted by $\tilde{f}$, is the element in $\mathfrak{H}^{\odot n}$ given by
\begin{align*}
	\tilde{f}= \frac{1}{n!}\sum_{\sigma\in S_n} \sum_{i_1,\ldots,i_n=1}^\infty a(i_1,\ldots,i_n)e_{i_{\sigma(1)}}\ldots e_{i_{\sigma(n)}}
\end{align*}  
where $S_n$ the symmetric group of order $n$. Now, let $\Lambda$ be the set of sequence $a=\brac{a_1,a_2,\ldots}$ such that only a finite number of $a_i$ are non-zero. Set $a!=\prod_{i=1}^\infty a_i!$ and $\abs{a}=\sum_{i=1}^\infty a_i$. For $a\in \Lambda$ such that $\abs{a}=n$, the multiple integral $I_n$ is defined via
\begin{align*}
I_n\brac{\operatorname{sym}\brac{\otimes_{i=1}^\infty e_i^{\otimes a_i}}}=\sqrt{a!}\prod_{i=1}^\infty H_{a_i}(Z(e_i))
\end{align*}
and 
\begin{align*}
	\operatorname{sym}\brac{\otimes_{i=1}^\infty e_i^{\otimes a_i}}=\frac{1}{n!}\sum_{\sigma\in S_n } \otimes_{k=1}^n
	 e_{j_{\sigma(k)} } \,, 
\end{align*}  
where  if all $a_i=0$ except $a_{i_1}, \cdots, a_{i_k}$,
then we denote $e_{j_1}= e_{i_1}, \cdots, e_{j_{a_{i_1}}}=e_{i_1}$, $e_{j_{a_{i_1}+1}}=e_{i_2}, 
\cdots,  e_{j_{a_{i_1}+a_{i_2}  }}=  e_{i_2},\cdots,
e_{j_{a_{i_1}+\cdots+  a_{i_{k-1}} +1 }}= e_{i_k}, \cdots,  e_{j_{n
  }}=e_{i_k}$.
In fact, $I_{n}$ is an isometry map between the space of
symmetric tensor products $\mathfrak{H}^{\odot n}$  equipped with the scaled norm
$\frac{1}{\sqrt{n!}}\norm{\cdot}_{\mathfrak{H}^{\otimes n}}$ and the
Wiener chaos of order $n$. This, combined with the decomposition \eqref{decomposition_hermite}, implies that any $F\in L^2(\Omega)$ can be expanded into an orthogonal sum of multiple Wiener integrals, that is
\begin{equation}
	\label{wienerdecompose} 
 F=\sum_{n=0}^\infty I_{n}(f_{n}),
\end{equation}
where $f_{n}\in \mathfrak{H}^{\odot n}$ are (uniquely determined) symmetric
functions and $I_{0}(f_{0})=\mathbb{E}\left(  F\right)$. 

Next, to extend the definition of multiple Wiener integrals to non-symmetric $f\in \frak{H}^{\otimes n}$, we simply define the action of $I_n$ on the basis elements of $\frak{H}^{\otimes n}$ as
\begin{align*}
    I_n\brac{\otimes_{i=1}^\infty e_i^{\otimes a_i}}	=I_n\brac{\operatorname{sym}\brac{\otimes_{i=1}^\infty e_i^{\otimes a_i}}}. 
\end{align*}

An important tool when dealing with multiple Wiener integrals is the {product formula}. For $f\in \mathfrak{H}^{\otimes n}$,
$g\in \mathfrak{H}^{\otimes m}$ and $0\leq r\leq m\wedge n$, the $r$-th contraction of $f$ and $g$  is
\begin{equation*}
	f\otimes _r g := \sum_{i_1,\ldots , i_r =1}^{\infty} \left\langle f,
	e_{i_1}\otimes \ldots \otimes e_{i_r}
	\right\rangle_{\mathfrak{H}^{\otimes r}} \otimes \left\langle g,
	e_{i_1}\otimes \ldots \otimes e_{i_r}
	\right\rangle_{\mathfrak{H}^{\otimes r}}. 
\end{equation*}
Then the product formula reads
\begin{equation}
	\label{prodformula}
	I_n(f)I_m(g)= \sum _{r=0}^{n\wedge m} r! \binom{n}{r}\binom{m}{r} I_{m+n-2r}(f\otimes _r g),
\end{equation}

We are now ready to introduce several important linear operators via the Gaussian chaos decomposition. The {Ornstein-Uhlenbeck semigroup} $\{P_t:t\geq 0\}$ is a semigroups of contraction operators that acts on a random variable $F$ with chaos decomposition \eqref{wienerdecompose} via
\begin{equation*}
	P_t F:=\sum_{n=0}^{\infty} e^{-nt}I_{n}(f_{n}).
\end{equation*}
Denote by $L$ the {infinitesimal generator} of our semigroups. The action of $L$ on a random variable $F$ of the form \eqref{wienerdecompose} and such
that $\sum_{n=1}^{\infty} n^{2}n! \norm{f_n}^2_{\mathfrak{H}^{\otimes n}}<\infty$ is given by
\begin{equation*}
	LF=-\sum_{n=1}^{\infty} nI_{n}(f_{n}).
\end{equation*}
Moreover, we can introduce a pseudoinverse $L^{-1}$ via spectral calculus by writing $L^{-1}F=-\sum_{n=1}^{\infty} \frac{1}{n}I_{n}(f_{n})$. Then, it immediately follows from our definition that
$LL^{-1}F = F - \E{F}$.

Let $\mathbb{D}^{k ,p }$ for $p>1$ and $k \in \mathbb{R}$ be the closure of
the set of polynomial random variables with respect to the norm
\begin{equation*}
	\Vert F\Vert _{k , p} =\Vert (I -L) ^{\frac{k }{2}} F \Vert_{L^{p} (\Omega )},
\end{equation*}
where $I$ denotes the identity operator. The spaces $\mathbb{D}^{k ,p }$ are known as {Sobolev-Watanabe spaces}.  Then, for any $p\geq 1$, the {Malliavin derivative} $D$ with respect to the Gaussian process $Z$ is a closable and continuous operator from $\mathbb{D}^{1,p} $ into $L^p(\Omega,\mathfrak{H})$, such that for $r\in \N$ and $f_n\in \mathfrak{H}^{\otimes n}$,
\begin{align*}
	DI_n(f_n)=n I_{n-1}(f_n).
\end{align*}
Alternatively, for a smooth random variable of the form $F=g(Z(h_1),\ldots , Z(h_n))$, where $g$ is a smooth function with compact support, and $h_i \in \mathfrak{H}$, $1 \leq i \leq n$, the Malliavin deriavative $D$ takes the form
\begin{equation*}
	DF:=\sum_{i=1}^{n}\frac{\partial g}{\partial x_{i}}(Z(h_1), \ldots , Z(h_n)) h_{i}.
\end{equation*}
From this, we can derive the {chain rule} of the Malliavin derivative $D$: for $\psi:\R^m\to \R$ a continuously differentiable function with bounded partial derivatives and $F=\brac{F_1,\ldots,F_m}$ a random vector such that $F_i\in \mathbb{D}^{1,p}$, it holds that
\begin{align}
	\label{mal_chain_rule}
	D\brac{\psi(F)}=\sum_{i=1}^m \partial_i \psi(F)DF_i.
\end{align}
The Malliavin derivative $D$ has a Hilbert adjoint $\delta$ which is commonly known as the
{divergence operator} or the {Skorokhod integral}. By Riesz's lemma, an element $u\in L^2(\Omega,\mathfrak{H})$
belongs to $\on{dom}(\delta)$ only if there exists a constant $C_u$
depending only on $u$ such that
$\abs{\E{\inner{DF,u}_\mathfrak{H}}}\leq C_u\norm{F}_{L^2(\Omega)}$
for any $F\in \mathbb{D}^{1,2}$. In this case, we have the {integration by parts formula} (or duality relation)
\begin{align}
	\label{intbypart_malliavin}
	\E{F\delta(u)}=\E{\inner{DF,u}_\mathfrak{H}}
\end{align}
for any $F\in \mathbb{D}^{1,2}$.
 
 $D,\delta,L$ are related via 	
\begin{equation*}
	LF = -\delta D F,
\end{equation*}
which applies to $F\in \on{dom}L =\mathbb{D}^{2,2}$ if and only if $F \in \on{dom}(\delta D)$ (that is $F \in \mathbb{D}^{1,2}$ and $DF \in\on{dom}(\delta)$).

At this point, let us state two important results that are respectively {hypercontractivity} of the Ornstein-Uhlenbeck semigroups $\{P_t:t\geq 0\}$ and the Poincar\'e  inequality. 
\begin{lemma}(\cite[Theorem 2.8.12]{nourdinpeccatibook})
\label{lemma_hypercontract}
    Assume $p>1,t>0$ and $q(t)=e^{2t}(p-1)+1$. If $F\in L^p(\Omega)$ then
\begin{align*}
	\norm{P_t F}_{q(t)}\leq \norm{F}_p. 
\end{align*}
\end{lemma}
In particular, this implies for $p>1$,  any two $L^p$ norms of a random variable $F$ inside a single Wiener chaos are equivalent, and consequently $F$ has moments of all orders.

\begin{lemma}(\cite[Lemma 5.3.7, Part 3]{nourdinpeccatibook})
\label{lemma_poincare}
    Let $p\geq 2$ be an even integer and $F\in \mathbb{D}^{1,p}$. Then
    \begin{align*}
        \E{\brac{F-\E{F}}^p}\leq (p-1)^{p/2} \E{\norm{DF}^p_\frak{H}}. 
    \end{align*}
    
\end{lemma}

Finally, let us informally relate the content of this section to the previous one. Malliavin calculus on Wiener space is in fact a Full Markov triple $\brac{E,\nu,\Gamma}$ where $E=\Omega$, $\nu$ is the laws of the Gaussian process $Z$ and $ \Gamma(F,G):=\inner{DF,DG}_\frak{H}$  for $F,G\in \mathbb{D}^{1,2}$. The chain rule of $D$ naturally induces a chain rule on $\Gamma$. The Markov diffusion chaos in this context are the Wiener chaos. One can refer to the book \cite{bouleauhirsch2010dirichlet} for more details on this connection.

\section{Density representation formulas}
\label{densityformulasection}
\subsection{Density representation formula for the Gamma distribution}
~\

Let $\brac{E,\nu,\Gamma}$ be a Full Markov triple equipped with
an infinitesimal generator $L$. In this section, we will derive a
density representation for random variables in $L^2(\nu)$.
This result will allow us to obtain density formulas for Gamma
distributed random variables in Corollary \ref{corollary_densitygamma}, which is the starting point leading to
our main results. In addition, we show in Remark \ref{remark_othertargets} how to get analogous representations for other targets beside Gamma
distributed random variables .

\begin{proposition}
\label{prop_densityrepgeneral}
Let $F\in \operatorname{Dom}(\mathcal{E}) $.
  Assume that
\begin{itemize}
    \item[(1)]  there exists some $p> 1$ such that  $\E{1/\Gamma(F)^p}<\infty$ and $\E{\Gamma(F)^{p/(p-1)}}<\infty$; 
    \item[(2)] $\tau(x)$ is the function defined by $\tau(x)=\E{\Gamma(F)|F=x}$.

    \item[(3)] the function 
    \begin{equation*}
f_0(x)=\int_a^x
      {1}/{\tau(y)}dy
\end{equation*}
for a given constant $a$ satisfies $f_0(F)\in \operatorname{Dom}(L)$.  
\end{itemize}
Then, $F$ admits a density on $\R$ the density function has the representation 
\begin{align*}
     p_F(x)=-\E{1_{\{x\leq F\}} Lf_0(F)}.
\end{align*}
\end{proposition}

\begin{proof}
%
%
Let $\psi \in C^\infty_c(\R)$ which consists of smooth functions with compact support. Then the diffusion property \eqref{diffusionproperty_gamma} of $\Gamma$ implies that 
\begin{align}
\label{firstrelation}
    \E{\psi'(F)}=\E{\frac{\Gamma\brac{F,\psi(F)}}{\Gamma(F)}}
\end{align}
{In particular,  H\"{o}lder inequality, the assumptions $\E{1/\Gamma(F)^p}<\infty$ and $\E{\Gamma(F)^{p/(p-1)}}<\infty$ indicate that $\E{\Gamma\brac{F,\psi(F)}/\Gamma(F)}\leq \E{\brac{\psi'(F)\Gamma(F)}^{p/(p-1)} }\E{\brac{{1}/{\Gamma(F)}}^p}<\infty$, so that the right hand side of \eqref{firstrelation} is well-defined.} Now, using \eqref{firstrelation}, the assumption (2) that $\Gamma(F)= \tau(F)$, the integration by part formula \eqref{intbyparts_gammaandL} and $f_0(F)\in \operatorname{Dom}(L)$, we have
\begin{align*}
  \E{\psi'(F)}=\E{\frac{\Gamma\brac{F,\psi(F)}}{\tau(F)}}=\E{\Gamma\brac{f_0(F),\psi(F)}}&=-
                \E{\psi(F)Lf_0(F)}.
\end{align*}

Next, via the fundamental theorem of calculus, we can write
\begin{align*}
    \E{\psi'(F)}&=-\E{\brac{\int_{-\infty}^F\psi'(x)dx }Lf_0(F)}=\int_\R -\E{\mathds{1}_{\{x\leq F\}} Lf_0(F)}\psi'(x)dx\,,
    \end{align*}
where in the above last equality Fubini's theorem is used.  This is legitimate  since $\psi \in C^\infty_c(\R)$ and $f_0(F)\in\operatorname{Dom}(L)\subseteq L^2(\nu)$ implies that $  \E{\abs{\mathds{1}_{\{x\leq F\}}  Lf_0(F)} }\leq \E{\abs{Lf_0(F)}^2 }^{1/2}<\infty$.

Let $a,b\in \R$. By an approximation argument, the above formula holds for $\psi(x)=\int_{-\infty}^x \mathds{1}_{[a,b]}(y)dy$ and leads to 
\begin{align*}
    \mathbb{P}(F\in [a,b])=\int_a^b -\E{\mathds{1}_{\{x\leq F\}} Lf_0(F)}dx.
\end{align*}
This implies $F$ admits a density with the representation $-\E{1_{\{x\leq F\}} Lf_0(F)}$. 
\end{proof}
More generally, we have the following result about the existence of a
density and its derivatives up to a certain order, along with representation formulas.

\begin{proposition}
\label{prop_densitygeneralformderivatives}
Let $F\in \operatorname{Dom}(\mathcal{E}) $.
 Assume that
\begin{itemize}
     \item[(1)]  there exists some $p> 1$ such that  $\E{1/\Gamma(F)^p}<\infty$;
    \item[(2)] $\tau(x)$ is the function defined by $\tau(x)=\E{\Gamma(F)|F=x}$;
    \item[(3)] the sequences $\{f_k\colon k\geq 0\}$ and $\{h_k
      \colon k\geq 1\}$ are well-defined in the following iterative way: 
    \begin{itemize}
        \item[(a)] for all $k\geq 0$, we have $f_k(F)\in \operatorname{Dom}(L)$;
        \item[(b)] there exists a constant $a$ such that
          $f_0(x)=\int_a^x {1}/{\tau(y)}dy$, and for $k\geq 1$,
    \begin{align}
    \label{def_fk}
        f_k(x)=\int_a^x \frac{h_k(y)}{\tau(y)}dy,
    \end{align}
where $h_k(x)=\E{Lf_{k-1}(F)|F=x}$;
\item[(c)] $\E{\brac{\Gamma(F)h_k(F)}^{p/(p-1)}}<\infty$. 
    \end{itemize}
    \end{itemize}

Then for every integer $k\geq 0$, the $k$-th derivative of the density of $F$ exists and has the representation
\begin{align*}
     p^{(k)}_F(x)=- \E{1_{\{x\leq F\}}Lf_k(F)}=- \E{1_{\{x\leq F\}}h_{k+1}(F)}. 
\end{align*}

\end{proposition}
\begin{proof}
%
%
We proceed by induction. For $k=0$, we refer to Proposition
\ref{prop_densityrepgeneral}. As our induction hypothesis, we assume
that for any integer $k\geq 2$,
\begin{align*}
    p_F^{(k-1)}(x)=-\E{\mathds{1}_{\{x\leq F\}}Lf_{k-1}(F) },
\end{align*}
where $f_{k-1}(x)=\int_a^x h_{k-1}(y)/\tau(y)dy$ and $h_{k-1}(x)=\E{Lf_{k-2}(F)|F=x}$. Now, suppose that
there exists  a function $h_k(x)$ satisfying $h_k(F)=Lf_{k-1}(F)$. 
 {For a smooth function $\psi$ with compact support, the diffusion property \eqref{diffusionproperty_gamma} 
of $\Gamma$ implies
\begin{align}
\label{firstrelation_derivativedensity}
    \E{\psi'(F) Lf_{k-1}(F)}=\E{\frac{\Gamma\brac{F,\psi(F)}}{\Gamma(F)}h_k(F) }. 
    \end{align}
In particular, H\"{o}lder inequality and $\E{1/\Gamma(F)^p},\E{\brac{\Gamma(F)h_k(F)}^{p/(p-1)}}<\infty$ indicate $\E{\frac{\Gamma\brac{F,\psi(F)}}{\Gamma(F)}h_k(F) }\leq \E{\brac{\psi'(F)\Gamma(F)h_k(F)}^{p/(p-1)}}\E{1/\Gamma(F)^p}<\infty$, so that the right hand side of \eqref{firstrelation_derivativedensity} is well-defined. 
 }

Next, define the function $f_k(x)$ as in \eqref{def_fk}. The assumption (2) implies   $  \tau(F)=\Gamma(F)$.   
Then we have, using \eqref{firstrelation_derivativedensity}, the diffusion property \eqref{diffusionproperty_gamma} of $\Gamma$,  the integration by part formula \eqref{intbyparts_gammaandL} and $f_k(F)\in \operatorname{Dom}(L)$, 
\begin{align*}
    \E{\psi'(F) Lf_{k-1}(F)}=\E{\frac{\Gamma\brac{F,\psi(F)}}{\tau(F)}h_k(F) }=\E{\Gamma(\psi(F),f_k(F))}=-\E{\psi(F)Lf_k(F)}.
\end{align*}

Now, via the fundamental theorem of calculus, we obtain

\begin{align*}
    \E{\psi'(F) Lf_{k-1}(F)}&=-\E{\brac{\int_{-\infty}^F \psi'(x)dx}  Lf_k(F) }=-\int_\R \psi'(x)\E{\mathds{1}_{\{x\leq F\}}Lf_k(F)}dx.
\end{align*}
The last equality is a consequence of Fubini's theorem, which can be applied since $\psi \in C^\infty_c(\R)$ and $f_k(F)\in\operatorname{Dom}(L)\subseteq L^2(\nu)\Rightarrow \E{\abs{\mathds{1}_{\{x\leq F\}}} Lf_k(F)}\leq \E{\abs{Lf_k(F)}^2 }^{1/2}<\infty$.

Finally, applying integration by parts for Riemann integrals
repeatedly yields
\begin{align}
\label{equation_densityderivativefirststep}
    \E{\psi'(F) Lf_{k-1}(F)}=(-1)^{k+1}\int_\R \psi^{(k+1)}(x) g(x)dx,
\end{align}
where $g^{(k)}(x):=\E{\mathds{1}_{\{x\leq F\}}Lf_k(F)}$. 

On the other hand, we can obtain a
different representation of $\E{\psi'(F) Lf_{k-1}(F)}$ by writing
\begin{align*}
 \E{\psi'(F) Lf_{k-1}(F)}=-\E{\Gamma\brac{\psi'(F),f_{k-1}(F)}}=-\E{\psi''(F)f'_{k-1}(F) \Gamma(F)},
\end{align*}
where we have once again used integration by parts of $\Gamma$ and the diffusion
property. Now, recalling that $f'_{k-1}(F)=h_{k-1}(F)/\tau(F)$ and
$\Gamma(F)=\tau(F)$, it follows that 
\begin{align*}
 \E{\psi'(F) Lf_{k-1}(F)}=-\E{\psi''(F)h_{k-1}(F)}=-\E{\psi''(F)Lf_{k-2}(F)},
\end{align*}
where the last equality results from the identity
$h_{k-1}(F)=Lf_{k-2}(F)$. Iterating this procedure $k$ times, we get
\begin{align*}
     \E{\psi'(F) Lf_{k-1}(F)}&=(-1)^{k-1}\E{\psi^{(k)}(F)Lf_0(F)
                               }=(-1)^k\E{\Gamma\brac{\psi^{(k)}(F),
                               f_0(F) }}.
\end{align*}
Applying the diffusion property for $\Gamma$ to both arguments yields
\begin{align*}
     \E{\psi'(F) Lf_{k-1}(F)}=(-1)^k\E{\psi^{(k+1)}(F)f_0'(F)\Gamma(F)}=(-1)^k\E{\psi^{(k+1)}(F)}.
\end{align*}
Exploiting the fact that $F$ admits a density $p_F$ under our induction hypothesis finally allows us to write
\begin{align}
    \label{equation_densityderivativesecondstep}
     \E{\psi'(F) Lf_{k-1}(F)}=(-1)^k\E{\psi^{(k+1)}(F)f_0'(F)\Gamma(F)}=(-1)^k\int_\R \psi^{(k+1)}(x) p_F(x)dx.
\end{align}
Combining \eqref{equation_densityderivativefirststep} and
\eqref{equation_densityderivativesecondstep} yields
\begin{align}
\label{equation_combinedstep}
    \E{\psi^{(k+1)}(F)}=\int_\R \psi^{(k+1)}(x) (-g(x))dx=\int_\R \psi^{(k+1)}(x) p_F(x)dx,
\end{align}
where we recall that $g(x)$ is a function satisfying $g^{(k)}(x)=\E{\mathds{1}_{\{x\leq F\}}Lf_k(F)}$. Now, let $a,b\in\R$.  By an approximation argument, \eqref{equation_combinedstep} holds for the function $\psi(x)=\int_{-\infty}^x\int_{-\infty}^{z_1}\ldots \int_{-\infty}^{z_k} \mathds{1}_{[a,b]}(z_{k+1})dz_{k+1}\ldots dz_2dz_1 $.  Then $\psi^{(k+1)}(x)=\mathds{1}_{[a,b]}(x)$  and
\begin{align*}
    \mathbb{P}(F\in [a,b])=\int_a^b -g(x)dx.
\end{align*}
This implies that $p_F(x)=-g(x)$ and hence 
\begin{align*}
    p^{(k)}_F(x)=-\E{\mathds{1}_{\{x\leq F\}}Lf_k(F)}.  
\end{align*}
\end{proof}

\begin{remark}
One can compare this result to a result in a recent work \cite[Proposition 2.1]{herry2023regularity}. Therein they obtain a general integration-by-part formula that also implies a density representation formula for elements of Markov diffusion chaos.
\end{remark}
Based on the previous result, we deduce a representation formula for the density of a Gamma random variable and the derivatives of said density function. 
\begin{corollary}
\label{corollary_densitygamma}
The density function of the Gamma random variable $\mathcal{G}_\alpha$
introduced in \eqref{formula_gammadensity_original} has the representation
\begin{align}
     p_{\mathcal{G}_\alpha}(x) =\begin{cases}
         \E{\mathds{1}_{\{\mathcal{G}_{\alpha}>x\}}\brac{1+\frac{1-\alpha}{\mathcal{G}_\alpha}}} &x\in (0,\infty);\\
        \E{\mathds{1}_{\{\mathcal{G}_{\alpha}< x\}}\brac{1+\frac{1-\alpha}{\mathcal{G}_\alpha}}}&x\in (-\infty,0] .
     \end{cases}\label{e.3.7}
\end{align}
In addition, for any integer $k\geq 1$, the
$k$-th derivative of the density function of $\mathcal{G}_\alpha$
admits the representation
\begin{align}
   p^{(k)}_{\mathcal{G}_\alpha}(x)=   \begin{cases}
         \E{\mathds{1}_{\{\mathcal{G}_\alpha>x\}}\nu_{k+1}(\mathcal{G}_\alpha) } &x\in (0,\infty);\\
        \E{\mathds{1}_{\{\mathcal{G}_\alpha< x\}}\nu_{k+1}(\mathcal{G}_\alpha) }&x\in (-\infty,0)\,,
             \end{cases}\label{e.3.8} 
\end{align}
where the function $\nu_{k+1}$ is given by
\begin{align}
    \label{def_nuk}
    \nu_{k }(y):=\begin{cases}
        (-1)^{k-1} \sum_{i=0}^{k }{k \choose i}\prod_{j=1}^i(j-\alpha)\frac{1}{y^i}& y\neq 0;\\
        0 & y=0, 
    \end{cases}
\end{align}
under the convention that $\prod_{k=1}^0 \alpha_k=1$
for any sequence $\alpha_k$. Furthermore, when $x=0$ and $\alpha>k+1$, we have
\begin{align*}
    p^{(k)}_{\mathcal{G}_\alpha}(0)=\E{\mathds{1}_{\{\mathcal{G}_\alpha<0\}}\nu_{k+1}(\mathcal{G}_\alpha) }=0. 
\end{align*}
\end{corollary}
\begin{proof}
%
%
Assume the full Markov triple $\brac{E,\nu,\Gamma}$ is a Laguerre
    structure per Example \ref{example_laguerre} and $\mathcal{G}_\alpha\sim \nu$. Recall that in this structure, we have 
    \begin{align}
    \label{laguerreoperator}
        Lf(\mathcal{G}_\alpha)=(\alpha-\mathcal{G}_\alpha)f'(\mathcal{G}_\alpha)+\mathcal{G}_\alpha f''(\mathcal{G}_\alpha); \qquad\qquad \Gamma(f(\mathcal{G}_\alpha))=\mathcal{G}_\alpha(f'(\mathcal{G}_\alpha))^2. 
    \end{align}

Next, we will apply Proposition
\ref{prop_densityrepgeneral} with $F=\mathcal{G}_\alpha$. Per \eqref{laguerreoperator} with $f$ being the identity map $x\mapsto x$, we have $ \Gamma(F)=\Gamma(\mathcal{G}_\alpha)=\mathcal{G}_\alpha$.  It is clearly that we can take any $p\in (0,\alpha)$ then $\E{(1/\Gamma(F))^p}=\E{(1/\mathcal{G}_\alpha)^p}<\infty$ and also $\E{\Gamma(F)^{p/(p-1)} }=\E{\mathcal{G}_\alpha^{p/(p-1)}}<\infty$.  Furthermore, $\tau(x)=\E{\Gamma(F)|F=x}=x$ so that
\begin{align*}
    f_0(x)=\int_1^x\frac{1}{\tau(y)}dy= \ln(x);  \qquad\qquad Lf_0(\mathcal{G}_\alpha)=-1-\frac{1-\alpha}{\mathcal{G}_\alpha}.
\end{align*}
Then Proposition \ref{prop_densitygeneralformderivatives} says 
\begin{align*}
     p_{\mathcal{G}_\alpha}(x)=\E{\mathds{1}_{\{\mathcal{G}_{\alpha}>x\}}\brac{1+\frac{1-\alpha}{\mathcal{G}_\alpha}}}, \qquad x>0. 
\end{align*}
In the case $x\leq 0$, we know $\mathcal{G}_\alpha>0$ almost surely so that 
\begin{align*}
    p_{\mathcal{G}_\alpha}(x)=\E{\mathds{1}_{\{\mathcal{G}_{\alpha}<x\}}\brac{1+\frac{1-\alpha}{\mathcal{G}_\alpha}}}=0, \qquad x\leq 0. 
\end{align*}
The combination of the previous two cases yield 
\begin{align*}
     p_{\mathcal{G}_\alpha}(x) =\begin{cases}
         \E{\mathds{1}_{\{\mathcal{G}_{\alpha}>x\}}\brac{1+\frac{1-\alpha}{\mathcal{G}_\alpha}}} &x\in (0,\infty);\\
        \E{\mathds{1}_{\{\mathcal{G}_{\alpha}< x\}}\brac{1+\frac{1-\alpha}{\mathcal{G}_\alpha}}}&x\in (-\infty,0] .
     \end{cases}
\end{align*}
The formula of the $k$-th derivative of the density function when $k\geq 1$ can be
deduced similarly via Proposition
\ref{prop_densitygeneralformderivatives}. We remind the reader that it was pointed out at \eqref{formula_gammadensityderivative_original} that $p^{(k)}_{\mathcal{G}_\alpha}(0)$ is well-defined only when $\alpha>k+1$. 

\end{proof}
\begin{remark}
\label{remark_recheckingdensityformula}
One can directly verify the formulas in Corollary \ref{corollary_densitygamma} by a simple integration by parts, that is 
\begin{align*}
    \int y^{\alpha-1}e^{-y}dy=-y^{\alpha-1}e^{-y}+(\alpha-1)\int y^{\alpha-2}e^{-y}dy,
\end{align*}
without referring to Proposition \ref{prop_densitygeneralformderivatives}. Indeed, 
\begin{align*}
    &\E{\mathds{1}_{\{\mathcal{G}_{\alpha}>x\}}\brac{1+\frac{1-\alpha}{\mathcal{G}_\alpha}}}
    =\frac{1}{\Gamma(\alpha)}\int_{x}^\infty \brac{1+\frac{1-\alpha}{y}}y^{\alpha-1}e^{-y}dy=\frac{1}{\Gamma(\alpha)}x^{\alpha-1}e^{-x}. 
\end{align*}
\end{remark}

\begin{remark}
  \label{remark_othertargets}
    In this remark, we demonstrate how Proposition \ref{prop_densityrepgeneral} can lead to density representations for general targets that are invariant probability measure of Markov processes. 
    
    Let $\{X_t:t\geq 0\}$ be a reversible It\^{o} diffusion process in $\R$ with the generator 
    \begin{align*}
        Lf(x)=\frac{1}{2}\sigma(x)^2f''(x)+b(x)f'(x). 
    \end{align*}
    Assume $\nu$ is its unique invariant probability measure that admits a density strictly positive on an interval $(\ell,u)$ with $-\infty\leq \ell<u\leq \infty$, and $0$ outside of $(\ell,u)$. Let $F$ be a random variable distributed as $\nu$. The carr\'{e} du champ operator associated with the process $\{X_t:t\geq 0\}$ acts on smooth random variables as $\Gamma(f(F))=\frac{1}{2}\sigma^2 (f'(F))^2$. Then using the notations from Proposition \ref{prop_densityrepgeneral}, the funtions $\tau$ and $f_0$ are respectively $\tau(x)=\frac{1}{2}\sigma(x)^2$ and $f_0(x)=\int_a^x \frac{2}{\sigma(y)^2}dy$ for some suitable $a$. It is straightforward to compute that  
    \begin{align*}
        Lf_0(F)= -\frac{2\sigma'(F)}{\sigma(F)}+\frac{2b(F)}{\sigma(F)^2}. 
    \end{align*}
    Then per Proposition \ref{prop_densityrepgeneral}, the density of the invariant probability measure $\nu$ has the representation 
    \begin{align*}
        p(x)=-2\E{1_{\{x\leq F \} } \brac{-\frac{\sigma'(F)}{\sigma(F)} +\frac{b(F)}{\sigma^2(F)}} }, \qquad x\in (\ell,u) \text{ and } F\sim \nu. 
    \end{align*}
    We emphasize that the above argument is not rigorous, and care has to be taken as in the proof of Corollary \ref{corollary_densitygamma}. We only give a short, informal argument here; and a rigorous version will be provided in a follow-up work. 
\end{remark}


\subsection{Density representation for multiple Wiener integrals}
~\

For a random variable $F\in \mathbb{D}^{1,2}$, let us write 
\begin{align*}
     w:=\norm{DF}^2_\mathfrak{H},\quad 
     u:=\frac{DF}{w}. 
\end{align*}
If $w>0$ almost surely (which is true for any multiple Wiener integral of an order greater than $0$), it is well-known (see \cite[Theorem 2.1.3]{nualart2006malliavin}) that $F$ admits a density. Moreover per \cite[Theorem 2.1.4]{nualart2006malliavin}, if we assume further that 
\begin{align*}
1/w\in \bigcap_{p\geq 1}L^p,
\end{align*}
then the density of $F$ is smooth and the $k$-th derivative of its density function has the representation
\begin{align}
\label{formula_densitymalcal}
    p^{(k)}_F(x)=(-1)^k\E{\mathds{1}_{\{F>x\}} G_{k+1}}.
\end{align}
The above terms $G_k$ are defined recursively by $G_0=1$ and $G_{k+1}=\delta\brac{G_ku}$. 

In the upcoming result, we will show how to decompose the above density formula in a way that is useful for getting Gamma convergence. Let us start by introducing some additional notations. Throughout the rest of the paper, unless specified otherwise, we assume $q\geq 2$ is a fixed even integer. For a pair of positive integers $l,m$ such that $1\leq m\leq l\leq q/2+1$ and $l+m\geq 3$, we define the constant
\begin{align}
\label{def_lambdakl}
     {\lambda(l,m)}:=\frac{(q-1)!\brac{q/2-l+1}!\brac{q/2-m+1}!}{(q-l-m+2)!(q/2)!^2}
\end{align}
and
\begin{align}
\label{def_Lambda}
     {\Lambda(l,m)}:=\lambda(l,m)D^lF\otimes_1D^mF-D^{l+m-2}F
\end{align}
which is a random variable taking value in $\mathfrak{H}^{\otimes l+m-2}$. Furthermore, we will write
\begin{align}
\label{def_Theta}
     \Theta:=\norm{DF}^2_\mathfrak{H}- q\brac{F+\alpha}=w-q\brac{F+\alpha} . 
\end{align}
The following lemma provides a representation formula for the density
(and its derivatives) of a multiple Wiener integral of an even order.
\begin{lemma}
\label{lemma_decompositiondensityforgammaconvergence}
Let $q$ be an even positive integer and $F$ be a random variable in
the $q$-th Wiener chaos. Assume further that $1/\norm{DF}^2_\frak{H}\in \cap_{p\geq 1}L^p$. Let $\alpha \in \R$ and $k\geq 0$ be a integer. Then, the $k$-th order derivative of the density function of $F+\alpha$ admits the representation 
\begin{align*}
    p^{(k)}_{F+\alpha}(x)=\E{\mathds{1}_{\{F+\alpha>x\}}\brac{\nu_{k+1}\brac{F+\alpha} +T_{k+1}}},
\end{align*}
where the function $\nu_{k+1}$ is defined at \eqref{def_nuk} while the term $T_{k+1}$ has the form
\begin{align}
    T_{k+1}&=a\Theta+\sum_{\substack{(l,m,n_1,\ldots,n_p,q_1,\ldots,q_r)\in I}}b_{l,m,n_1,\dots,n_p,q_1,\ldots,q_r}\nonumber\\
    &\qquad\qquad\qquad\qquad \inner{\Lambda(l,m)\otimes D^{n_1}F\otimes \ldots \otimes D^{n_p}F,D^{q_1}F\otimes\ldots\otimes D^{q_r}F}_{\mathfrak{H}^{l+m-2+\sum_{i=1}^pn_i}}\label{e.3.10} 
\end{align}
where $\Lambda(l,m)$ and $\Theta$ are respectively defined at
\eqref{def_Lambda} and \eqref{def_Theta}. The index set $I$ is given by
\begin{align*}
    I&=\bigg\{ \brac{l,m,n_1,\ldots,n_p,q_1,\ldots,q_r} \colon 1\leq l,m\leq q/2+1;\ 3\leq l+m \leq k+3;\\
    &\quad\quad\quad 0\leq n_i\leq (k+3)\wedge q\ \mbox{for}\ 1\leq i\leq p;\ 0\leq
      q_i\leq (k+3)\wedge q\ \mbox{for}\ 1\leq i\leq r;\\
  & \qquad\qquad\qquad\qquad\qquad\qquad\qquad\qquad\qquad\qquad\qquad\quad l+m-2+\sum_{i=1}^pn_i=\sum_{i=1}^r q_i\bigg\}.
\end{align*}
Furthermore, $a$ and $b_{l,m,n_1,\ldots,n_p,q_1,\ldots,q_r}$ are real
polynomials in the variables 
\begin{align}
\label{list_variablesinaandb}
     F,\ \frac{1}{F+\alpha},\ \frac{1}{w},\ \inner{D^{\tilde{l}}F\otimes_{1}D^{\tilde{m}}F\otimes D^{\tilde{n}_1}F\otimes \ldots \otimes  D^{\tilde{n}_p}F, D^{\tilde{q}_1}F\otimes\ldots\otimes D^{\tilde{q}_r}F}_{\mathfrak{H}^{\tilde{l}+\tilde{m}-2+\sum_{i=1}^p \tilde{n}_i}}
\end{align}
such that the following set of conditions, which we shall call $\operatorname{\textbf{Con}}-(k+1)$, is satisfied. 
  \begin{enumerate}
\item[(i)] $1\leq \tilde{l},\tilde{m}\leq q/2+1$ and $3\leq \tilde{l}+\tilde{m}\leq k+3$;
\item[(ii)] $0\leq \tilde{n}_i\leq (k+3)\wedge q$ for $1\leq i\leq p$ and $0\leq \tilde{q}_i\leq (k+3)\wedge q$ for $1\leq i\leq r$;
\item[(iii)] $\tilde{l}+\tilde{m}-2+\sum_{i=1}^p \tilde{n}_i=\sum_{i=1}^{r} \tilde{q}_i$.
\end{enumerate}
Moreover, when factors $1/w^i$ or $1/(F+\alpha)^j$ appear in the polynomials $a$ or $b_{l,m,n_1,\ldots,n_p,q_1,\ldots,q_r}$, then they satisfy
\begin{enumerate}
    \item[(iv)] $i\leq 2k+2$ and $j\leq k+1$. 
\end{enumerate}
\end{lemma}
\begin{proof}
Let us start with the case $k=0$. First we set $\nu_1(y)=1+\frac{1-\alpha}{y}$. Then using formula \eqref{formula_densitymalcal}, $p_{F+\alpha}(x)$ can be written as 
\begin{align*}
     p_{F+\alpha}(x)=\E{\mathds{1}_{\{F+\alpha>x\}}\delta(u)}=\E{\mathds{1}_{\{F+\alpha>x\}}\brac{\nu_{1}\brac{F+\alpha} +T_{1}}},
\end{align*}
where $T_1$ is  defined  by \eqref{e.3.10} with $k=0$, namely, 
\begin{align}
\label{equation_T1}
    T_1&=\delta(u)-\brac{1+\frac{1-\alpha}{F+\alpha}}\nonumber\\
&=\frac{qF}{w}+\frac{1}{w^2}\inner{2D^2F\otimes_1DF,DF}_\mathfrak{H} -\brac{1+\frac{1-\alpha}{F+\alpha}}\nonumber\\
&=\brac{\frac{q(F+\alpha)}{w}-1}+(1-\alpha)\brac{\frac{q}{w}-\frac{1}{F+\alpha}}+\frac{1}{w^2}\inner{2D^2F\otimes_1DF-qDF,DF}_\mathfrak{H}\nonumber\\
    &=\brac{-\frac{1}{w}-\frac{(1-\alpha)}{w(F+\alpha)} }\Theta +\frac{q}{w^2}\inner{\Lambda(2,1),DF}_\mathfrak{H}. 
\end{align}
To get the second line, we use \cite[Proposition 1.3.3]{nualart2006malliavin} and to get the last line,   we use  $\Lambda(2,1)=(2/q)D^2F\otimes_1 DF-DF$. Thus, the
statement of the lemma holds true  for $k=0$. We now assume that the statement of the lemma holds true for the $k$-th derivative of $p_{F+\alpha}$. Then, by the integration by part formula \eqref{intbypart_malliavin}, the $(k+1)$-th derivative satisfies
\begin{align}
\label{formula_k+1derivativedecompose_initialstep}
   p^{(k+1)}_{F+\alpha}(x)
    =&\frac{d}{dx}p^{(k)}_{F+\alpha}(x) = \frac{d}{dx}\E{\mathds{1}_{\{F+\alpha>x\}}\brac{\nu_{k+1}\brac{F+\alpha} +T_{k+1}}}\nonumber \\
    =& \E{\langle D\mathds{1}_{\{F+\alpha>x\}}, u\rangle \brac{\nu_{k+1}\brac{F+\alpha} +T_{k+1}}}\nonumber \\
    =&-\E{\mathds{1}_{\{F+\alpha>x\}}\brac{\delta\brac{\nu_{k+1}(F+\alpha)u}+\delta\brac{T_{k+1}u} }}.
\end{align}
We split our calculation into two parts. In the first part, let us
expand the term $\delta\brac{\nu_{k+1}(F+\alpha)u}$. Using
\cite[Proposition 1.3.3]{nualart2006malliavin}, we can write
\begin{align*}
  \delta\brac{\nu_{k+1}(F+\alpha)u}&=\frac{\nu_{k+1}(F+\alpha)qF}{w}+\frac{\nu_{k+1}(F+\alpha)}{w^2}\inner{2D^2F\otimes_1DF,DF}_\mathfrak{H}\\
  &\qquad -\frac{1}{w}\inner{D\nu_{k+1}(F+\alpha), DF}_\mathfrak{H}\\
    &=\nu_{k+1}(F+\alpha)\brac{1+\frac{1-\alpha}{F+\alpha} }-\brac{\frac{F\nu_{k+1}(F+\alpha)}{F+\alpha}-\frac{qF\nu_{k+1}(F+\alpha)}{w} }\\
    & \qquad -\brac{\frac{\nu_{k+1}(F+\alpha)}{F+\alpha}-\frac{q\nu_{k+1}(F+\alpha)}{w}}-\nu'_{k+1}(F+\alpha)\\
    &\qquad +\brac{\frac{\nu_{k+1}(F+\alpha)}{w^2}\inner{2D^2F\otimes_1DF,DF}_\mathfrak{H} -\frac{q\nu_{k+1}(F+\alpha)}{w} }\\
    &=\nu_{k+1}(F+\alpha)\brac{1+\frac{1-\alpha}{F+\alpha} }-\frac{F\nu_{k+1}(F+\alpha)}{(F+\alpha)w}(w-q(F+\alpha)) \\
    &-\frac{\nu_{k+1}(F+\alpha)}{w(F+\alpha)}(w-q(F+\alpha))-\nu'_{k+1}(F+\alpha)\\
    &+\brac{\frac{\nu_{k+1}(F+\alpha)}{w^2}\inner{2D^2F\otimes_1DF,DF}_\mathfrak{H}-\frac{\nu_{k+1}(F+\alpha)}{w^2}q\inner{DF,DF}_\frak{H}}. 
\end{align*}
Then by using the definition of $\Theta=w-q\brac{F+\alpha}$ and $\Lambda(2,1)=(2/q)D^2F\otimes_1 DF-DF$, we obtain
    \begin{align*}
  \delta\brac{\nu_{k+1}(F+\alpha)u}  &=\nu_{k+1}(F+\alpha)\brac{1+\frac{1-\alpha}{F+\alpha} }-\nu'_{k+1}(F+\alpha)\\
    & -\brac{\frac{F\nu_{k+1}(F+\alpha)}{(F+\alpha)w}+\frac{\nu_{k+1}(F+\alpha)}{(F+\alpha)w}}\Theta+\frac{q\nu_{k+1}(F+\alpha)}{w^2}\inner{\Lambda(2,1),DF}_\mathfrak{H}. 
\end{align*}
It is worth noticing here that
\begin{align*}
    \nu'_{k+1}(y)=(-1)^{k+1}\sum_{i=1}^{k+2}{k+1\choose i-1}\prod_{j=1}^{i-1}(j-\alpha)\frac{i-1}{y^i}. 
\end{align*}
Therefore,
\begin{align*}
  \delta\brac{\nu_{k+1}(F+\alpha)u}&=f(F+\alpha)-\brac{\frac{F\nu_{k+1}(F+\alpha)}{(F+\alpha)w}+\frac{\nu_{k+1}(F+\alpha)}{(F+\alpha)w}}\Theta\\
  &\qquad\qquad\qquad+\frac{q\nu_{k+1}(F+\alpha)}{w^2}\inner{\Lambda(2,1),DF}_{\mathfrak{H}},
\end{align*}
where 
\begin{align*}
    f(y)=\nu_{k+1}(y)+\frac{1-\alpha}{y}\nu_{k+1}(y)-\nu'_{k+1}(y).
\end{align*}
We will now check that $\nu_{k+2}(y)=-f(y)$. Recall the convention $\prod_{j=1}^0(j-\alpha)=1$. We have
\begin{align*}
    f(y)&= (-1)^k\sum_{i=0}^{k+1}{k+1\choose i}\prod_{j=1}^i(j-\alpha)\frac{1}{y^i}+(-1)^k\sum_{i=1}^{k+2}{k+1\choose i-1}\prod_{j=1}^{i-1}(j-\alpha)\frac{1-\alpha}{y^i}\\
    &\qquad -(-1)^{k+1}\sum_{i=1}^{k+2}{k+1\choose i-1}\prod_{j=1}^{i-1}(j-\alpha)\frac{i-1}{y^i}\\
    &=(-1)^k+ (-1)^k\sum_{i=1}^{k+1}{k+1\choose i}\prod_{j=1}^i(j-\alpha)\frac{1}{y^i}+(-1)^k{\prod_{j=1}^{k+1}(j-\alpha)\frac{1-\alpha+(k+2)-1}{y^{k+2}}} \\
    &\qquad + (-1)^{k}\sum_{i=1}^{k+1}{k+1\choose i-1}\prod_{j=1}^{i-1}(j-\alpha)\frac{1-\alpha+i-1}{y^i}\\
        &=(-1)^k+ (-1)^k\sum_{i=1}^{k+1}{k+1\choose i}\prod_{j=1}^i(j-\alpha)\frac{1}{y^i}+(-1)^k{\prod_{j=1}^{k+2}(j-\alpha)\frac{1}{y^{k+2}}} \\
    &\qquad + (-1)^{k}\sum_{i=1}^{k+1}{k+1\choose i-1}\prod_{j=1}^{i}(j-\alpha)\frac{1}{y^i}.
\end{align*}
Now by combining the second and fourth terms on the right hand side and using the identity ${k+2\choose i}={k+1\choose i}+{k+1\choose i-1}$, we arrive at 
\begin{align*}
    f(y)&=(-1)^k\brac{1+\prod_{j=1}^{k+2}(j-\alpha)\frac{1}{y^{k+2}} }+(-1)^k\sum_{i=1}^{k+1}{k+2\choose i}\prod_{j=1}^i(j-\alpha)\frac{1}{y^i}\\
    &=-\nu_{k+2}(y).
\end{align*}

This leads to
\begin{align}
\label{proof_decompositionfirststep}
  \delta\brac{\nu_{k+1}(F+\alpha)u}&=-\nu_{k+2}(F+\alpha)-\brac{\frac{F\nu_{k+1}(F+\alpha)}{(F+\alpha)w}+\frac{\nu_{k+1}(F+\alpha)}{(F+\alpha)w}}\Theta\nonumber\\
  &\qquad  +\frac{q\nu_{k+1}(F+\alpha)}{w^2}\inner{\Lambda(2,1),DF}_\mathfrak{H}. 
\end{align}
Consider the previous expansion in the context of equation
\eqref{formula_k+1derivativedecompose_initialstep}. The second part of
the proof will be dedicated to showing that 
\begin{align}
\label{def_mathcalA}
    \mathcal{A}&=-\brac{\frac{F\nu_{k+1}(F+\alpha)}{(F+\alpha)w}+\frac{\nu_{k+1}(F+\alpha)}{(F+\alpha)w}}\Theta+\frac{q\nu_{k+1}(F+\alpha)}{w^2}\inner{\Lambda(2,1),DF}_\mathfrak{H}+\delta\brac{T_{k+1}u}\nonumber\\
&=\mathcal{A}_1+\mathcal{A}_2+\delta\brac{T_{k+1}u}
\end{align}
is the quantity $T_{k+2}$ as stated in the statement of the
lemma. Verifying that $\mathcal{A}_1$ and $\mathcal{A}_2$ appearing in
the above equality are real polynomials that satisfy
$\operatorname{\textbf{Con}}-(k+2)$ is straightforward. Let us focus
on the term $\delta\brac{T_{k+1}u}$,  where $T_{k+1}$ is given by \eqref{e.3.10}. We have
\begin{align}
\label{equation_deltaTu}
    \delta\brac{T_{k+1}u}&=\delta\brac{a\Theta u}\nonumber\\
    &\quad +\sum_{ I}\delta\brac{b_{l,m,n_1,\ldots,n_p,q_1,\ldots,q_r}\inner{\Lambda(l,m)\otimes \bigotimes_{i=1}^p D^{n_i}F,\bigotimes_{j=1}^r D^{q_j}F}_{\mathfrak{H}^{\otimes \sum_{i=1}^r q_i}}u}.
\end{align}
Regarding the first term on the right hand side, one can apply \cite[Proposition 1.3.3]{nualart2006malliavin} and Lemma \ref{lemma_derivativeofLambda} stated at the end of this section to get
\begin{align}
\label{equation_delta_athetau}
    \delta\brac{a\Theta u}&=\Theta\delta(au)-\inner{D\Theta,au}_\mathfrak{H}\nonumber\\
    &=\brac{\frac{aqF}{w}+\frac{1}{w^2}\inner{2D^2F,DF^{\otimes 2}}_{\mathfrak{H}^{\otimes 2}}-\inner{Da,u}_\mathfrak{H}}\Theta-\frac{qa}{w}\inner{\Lambda(2,1), DF}_\mathfrak{H}. 
\end{align}
We will now show that $\inner{Da,u}_\mathfrak{H}$ on the right hand side is a real polynomial satisfying $\operatorname{\textbf{Con}}-(k+2)$. The induction hypothesis says that any term in the polynomial $a$ has the form 
\begin{align*}
 \mathcal{T}= C\frac{1}{w^i}\frac{1}{(F+\alpha)^j}F^v\prod_{s\in I}\inner{D^{{l}}F\otimes_{1}D^{{m}}F \otimes\bigotimes_{i=1}^p D^{n_i}F,\bigotimes_{j=1}^r D^{q_j}F}_{\mathfrak{H}^{\otimes \sum_{i=1}^r q_i}}^{V_s}.
\end{align*}
Here, $i,j,v$ and $\{V_s:s\in I\}$ are non-negative integers for which
$i\leq 2k+2$ and $j\leq k+1$ and $C$ is some real constant. Hence, the
verification that $\inner{Da,u}_\mathfrak{H}$ in
\eqref{equation_delta_athetau} satisfies
$\operatorname{\textbf{Con}}-(k+2)$ reduces to verifying that the term
$\inner{D\mathcal{T},u}_\mathfrak{H}$ satisfies
$\operatorname{\textbf{Con}}-(k+2)$. By Leibniz's product rule, we have
\begin{align*}
    &\inner{D\mathcal{T},u}_\mathfrak{H}\\
    &=C\inner{D\frac{1}{w^i},u}_\mathfrak{H}\frac{1}{(F+\alpha)^j}F^v\prod_{s\in I}\inner{D^{{l}}F\otimes_{1}D^{{m}}F \otimes \bigotimes_{i=1}^p D^{n_i}F,\bigotimes_{j=1}^r D^{q_j}F}_{\mathfrak{H}^{\otimes \sum_{i=1}^r q_i}}^{V_s}\\
    &\quad +C\frac{1}{w^i}\inner{D\frac{1}{(F+\alpha)^j},u}_\mathfrak{H}F^v\prod_{s\in I}\inner{D^{{l}}F\otimes_{1}D^{{m}}F \otimes\bigotimes_{i=1}^p D^{n_i}F,\bigotimes_{j=1}^r D^{q_j}F}_{\mathfrak{H}^{\otimes \sum_{i=1}^r q_i}}^{V_s}\\
    &\quad +C\frac{1}{w^i}\frac{1}{(F+\alpha)^j}\inner{DF^v,u}_\mathfrak{H}\prod_{s\in I}\inner{D^{{l}}F\otimes_{1}D^{{m}}F\otimes \bigotimes_{i=1}^p D^{n_i}F,\bigotimes_{j=1}^r D^{q_j}F}_{\mathfrak{H}^{\otimes \sum_{i=1}^r q_i}}^{V_s}\\
    &\quad +C\frac{1}{w^i}\frac{1}{(F+\alpha)^j}F^v\inner{D\prod_{s\in I}\inner{D^{{l}}F\otimes_{1}D^{{m}}F \otimes \bigotimes_{i=1}^p D^{n_i}F,\bigotimes_{j=1}^r D^{q_j}F}_{\mathfrak{H}^{\otimes \sum_{i=1}^r q_i}}^{V_s},u}_\mathfrak{H}. 
\end{align*}

Observe that for positive integers $i,j,v$ and $V_s$ with $i\leq 2k+2$
and $j\leq k+1$, per the chain rule \eqref{mal_chain_rule} one has
\begin{align}
  \inner{D\frac{1}{w^i},u}&=-\frac{2i}{w^{i+2}}\inner{D^2F\otimes_1DF,DF}_\mathfrak{H}\label{deg_i}
\end{align}
as well as
\begin{align}
    \inner{D\frac{1}{(F+\alpha)^j},u}&=-\frac{j}{(F+\alpha)^{j+1}}; \label{deg_j}
\end{align}
and also
\begin{align*}
\inner{DF^v,u}&=vF^{v-1}.
\end{align*}
Finally, we also have
\begin{align*}
    &\inner{D\inner{D^{{l}}F\otimes_{1}D^{{m}}F \otimes\bigotimes_{i=1}^p D^{n_i}F,\bigotimes_{j=1}^r D^{q_j}F}^{V_{s}}_{\mathfrak{H}^{\otimes \sum_{i=1}^r q_i}}, u}_\mathfrak{H}\\
     &\qquad\qquad =V_s \inner{D^{{l}}F\otimes_{1}D^{{m}}F\otimes \bigotimes_{i=1}^p D^{n_i}F,\bigotimes_{j=1}^r D^{q_j}F}_{\mathfrak{H}^{\otimes \sum_{i=1}^r q_i}}^{V_s-1}\\
     &\qquad\qquad\qquad\times \frac{1}{w}\Bigg(\Bigg\langle D^{{l}+1}F\otimes_{1}D^{{m}}F\otimes\bigotimes_{i=1}^p D^{{n}_i}F 
     +D^{{l}}F\otimes_{1}D^{{m}+1}F\otimes\bigotimes_{i=1}^p D^{{n}_i}F \\
     &\qquad\qquad\quad +D^{{l}}F\otimes_{1}D^{{m}}F\otimes\brac{\sum_{i=1}^p D^{{n}_i+1}F\otimes\bigotimes_{\substack{i'=1\\i'\neq i}}^p D^{{n}_{i'}}F } ,DF \bigotimes_{j=1}^rD^{{q}_j}F\Bigg\rangle_{\mathfrak{H}^{\otimes 1+ \sum_{i=1}^r q_i}}\\
     &\qquad\qquad\quad +\Bigg\langle DF\otimes D^{{l}}F\otimes_{1}D^{{m}}F\otimes\bigotimes_{i=1}^p D^{{n}_i}F, {\sum_{j=1}^r D^{{q}_j+1}F\otimes\bigotimes_{\substack{j'=1\\j'\neq j}}^r D^{{q}_{j'}}F }\Bigg\rangle_{\mathfrak{H}^{\otimes 1+\sum_{i=1}^r q_i}}\Bigg).
\end{align*}
Thanks to these identities, one can see
$\inner{D\mathcal{T},u}_\mathfrak{H}$ is a real polynomial satisfying
parts $(i)$ through $(iii)$ of $\operatorname{\textbf{Con}}-(k+2)$. To check part $(iv)$, let us suppose $\mathcal{T}$ contains the factors $1/(F+\alpha)^{k+1}$ and $1/w^{2k+2}$, then the calculations at \eqref{deg_i} and \eqref{deg_j} indicate that $\inner{D\mathcal{T},u}_\mathfrak{H}$ contains the factors $1/(F+\alpha)^{k+2}$ and $1/w^{2k+4}$.
    
Next, the quantity
$\delta\brac{b_{l,m,n_1,\ldots,n_p,q_1,\ldots,q_r}\Lambda(l,m)u}$ on
the right hand side of \eqref{equation_deltaTu} can be analysed in the
same way as $\delta\brac{a\Theta u}$ at
\eqref{equation_delta_athetau}.

Combining the previous expansion of $\delta\brac{T_{k+1}u}$ together
with \eqref{formula_k+1derivativedecompose_initialstep} and
\eqref{proof_decompositionfirststep} yields
\begin{align*}
        p^{(k+1)}_{F+\alpha}(x)=\E{\mathds{1}_{\{F+\alpha>x\}}\brac{\nu_{k+2}\brac{F+\alpha} +T_{k+2}}},
\end{align*}
where $T_{k+2}=\mathcal{A}$ defined at \eqref{def_mathcalA} is a real
polynomial satisfying $\operatorname{\textbf{Con}}-(k+2)$, which
concludes the proof.
\end{proof}
The following lemma provides a technical link between the quantities
$\Lambda$ and $\Theta$ defined at \eqref{def_Lambda} and
\eqref{def_Theta}, respectively.
\begin{lemma}
\label{lemma_derivativeofLambda}
Recall the quantities $\Lambda(k,l)$ and $\Theta$ defined respectively at \eqref{def_Lambda} and \eqref{def_Theta}. It holds that
    \begin{align*}
        D\Theta&=q\Lambda(2,1)
        \end{align*}
and for $k,l\in\N$ such that $k+l\geq 3$,
\begin{align*}
    D\Lambda(k,l)=\frac{\lambda(k,l)}{\lambda(k+1,l)}\Lambda(k+1,l)+\frac{\lambda(k,l)}{\lambda(k,l+1)}\Lambda(k,l+1).     
\end{align*}
\end{lemma}
\begin{proof}
Proving the first identity is straightforward. Indeed, 
    \begin{align*}
        D\Theta=D\brac{\norm{DF}^2-q(F+\alpha)}=2D^2F\otimes_1DF-qDF=q\Lambda(2,1).
    \end{align*}
For the second identity, it is sufficient to verify the relation
\begin{align*}
    \frac{1}{\lambda(k+1,l)}+\frac{1}{\lambda(k,l+1)}=\frac{1}{\lambda(k,l)}
\end{align*}
since if it is true, then we can write
\begin{align*}
    D\Lambda(k,l)&=\lambda(k,l)D^{k+1}F\otimes_1 D^lF+\lambda(k,l)D^{k}F\otimes_1 D^{l+1}F-D^{k+l-1}F\\
    &=\frac{\lambda(k,l)}{\lambda(k+1,l)}\brac{ \lambda(k+1,l)D^{k+1}F\otimes_1 D^lF-D^{k+l-1}F}\\
    &\quad +\frac{\lambda(k,l)}{\lambda(k,l+1)}\brac{ \lambda(k,l+1)D^{k}F\otimes_1 D^{l+1}F-D^{k+l-1}F}.
\end{align*}
This relation does indeed hold as we have
\begin{align*}
  \frac{1}{\lambda(k+1,l)}+\frac{1}{\lambda(k,l+1)}&=\frac{(q-k-l+1)!(q/2)!^2}{(q-1)!\brac{q/2-k}!\brac{q/2-l+1}!}\\
  &\quad +\frac{(q-k-l+1)!(q/2)!^2}{(q-1)!\brac{q/2-k+1}!\brac{q/2-l}!}\\
                                                   &=\frac{(q-k-l+1)!(q/2)!^2\brac{q-k-l+2}}{(q-1)!\brac{q/2-k+1}!\brac{q/2-l+1}!}\\
  &=\frac{1}{\lambda(k,l)},
\end{align*}
which concludes the proof.
\end{proof}

\section{Stein's method}~\
\label{sectionsteinmethod}
In this section, we introduce Stein's method for Gamma approximation,
and we derive several important estimates that will play a key role in the proof of Theorems \ref{theorem_fourthmoment_singlechaos_0thderivative}, \ref{theorem_fourthmoment_singlechaos_kthderivative} and \ref{theorem_gammaconvergence_sumofchaos}. Readers can refer to 
\cite{NP09b,NP09main, DP18} for a more detailed exposition of Stein's
for Gamma approximation. 

For a function $h$ belonging to some appropriate class of functions
$\mathscr{H}$, a choice for the Stein equation for Gamma approximation can be 
\begin{align}
\label{equation_gammastein}
    yf'(y)+(\alpha-y)f(y)=h(y)-\E{h(\mathcal{G}_{\alpha})},
\end{align}
where $\mathcal{G}_\alpha,\alpha>0$ is a Gamma distributed random
variable as specified at \eqref{formula_gammadensity_original}. Let us
assume that for every $h\in \mathscr{H}$, a solution $f_h$ to the above ordinary differential equation exists. Then, for any random variable $X$ such that $\E{\abs{h(X)}}<\infty$, we can write
\begin{align*}
    \sup_{h\in\mathscr{H}}\abs{ \E{h(X)}-\E{h(\mathcal{G}_{\alpha})}}\leq   \sup_{h\in\mathscr{H}}\abs{\E{Xf_h'(X)+(\alpha-X)f_h(X)}}.
\end{align*}
The advantage of the above Stein  inequality is that  the right hand side 
depends only on the variable $X$,   involving no longer $\mathcal{G}_{\alpha}$.  
If we take $\mathscr{H}$ to be the collection of all indicate  functions
of Borel sets of $\R$, then the above equation gives an upper bound on
the total variation distance between $X$ and $\mathcal{G}$. If we
instead take $\mathscr{H}$ to be the collection of all 1-Lipschitz
function on $\R$, then the above inequality yields an upper bound on
the Wasserstein distance between $X$ and $\mathcal{G}$. In the present paper, the chosen class of
functions $\mathscr{H}$ is going to be a special one that comes from the
representation of the Gamma density in Corollary
\ref{corollary_densitygamma}.

We begin by a proposition about the properties of the solution to the
Stein equation \eqref{equation_gammastein} for specific test functions
$h \in \mathscr{H}$ that is suitable for approximating $p_{\mathcal{G}_\alpha}(x)$ for $x>0$.

\begin{proposition}
\label{prop_steinmethod_x>0}
Let $x>0$ and the test function $h$ in the ordinary differential equation stated at \eqref{equation_gammastein} be
\begin{align*}
    h_{k+1}(y)=\mathds{1}_{\{y>x\}}\nu_{k+1}(y), \quad y\in \R
\end{align*}
where $\nu_{k+1}$, $k\geq 0$, is defined at \eqref{def_nuk}. Then, a
solution $f_{h_{k+1}}: \R\setminus \{ 0\}\to \R$ to equation \eqref{equation_gammastein} exists and satisfies, for every $y\neq 0$,
\begin{align}
\label{estimate_derivativesolutionstein_x>0}
 \max\left\{ \abs{f_{h_{k+1}}(y)},  \abs{f'_{h_{k+1}}(y)}\right\}\leq d_1(x)\sum_{i=1}^{\ceil{\alpha}-1} \abs{y}^{\alpha-i}+\frac{d_2(x)}{\abs{y}}+d_3(x). 
\end{align}
Here the factors $d_1(x),d_2(x)$ and $d_3(x)$ are positive and finite
for every $x>0$. 

Consequently if $F$ is a random variable in
$\mathbb{D}^{1,4}$ such that $\E{F}=0$, $\E{F^2}=\alpha$ and $\E{\frac{1}{(F+\alpha)^2}}<\infty$, $\E{(F+\alpha)^{2(\alpha-i)}}<\infty$ for $1\leq i\leq \ceil{\alpha}-1$ then, for
every $x >0$, we have the pointwise estimate
    \begin{align}
    \label{esimate_steinmethod_x>0}
      &\abs{\E{\mathds{1}_{\{F+\alpha>x\}}\nu_{k+1}(F+\alpha)}- \E{\mathds{1}_{\{\mathcal{G}_{\alpha}>x\}}\nu_{k+1}(\mathcal{G}_{\alpha})}}\nonumber\\
      &\qquad \leq  \brac{d_1(x)\sum_{i=1}^{\ceil{\alpha}-1}\E{(F+\alpha)^{2(\alpha-i)}}^{1/2}+ {d_2(x)}\E{\frac{1}{(F+\alpha)^2}}^{1/2}+{d_3(x)}}\nonumber\\
      &\hspace{17em}\times\E{\brac{{{F+\alpha}} -\inner{DF,-DL^{-1}F}_\mathfrak{H}}^2}^{1/2}.
    \end{align}
\end{proposition}
\begin{proof}
 A solution to the differential equation \eqref{equation_gammastein} is
\begin{align*}
    f_{h_{k+1}}(y)&= e^yy^{-\alpha}\int_0^y\brac{h_{k+1}(t)-\E{h_{k+1}(\mathcal{G}_\alpha)}}t^{\alpha-1}e^{-t}dt. 
\end{align*}
This solution is well-defined for all $\alpha>0$ and $y\in \R\setminus \{ 0\}$. Indeed, $e^yy^{-\alpha}<\infty$ for $y\in \R\setminus \{ 0\}$. Meanwhile, $x>0$ so that 
\begin{align*}
    \nu_{k+1}(t)=(-1)^k\sum_{i=0}^{k+1}{k+1\choose i}\prod_{j=1}^i(j-\alpha)\frac{1}{t^i}. 
\end{align*}
is well-defined on $[x,\infty)$, and hence for all $y\in \R$, 
\begin{align*}
    \int_0^y h_{k+1}(t)t^{\alpha-1}e^{-t}dt=\int_0^y \mathds{1}_{\{t>x \}} \nu_{k+1}(t)t^{\alpha-1}e^{-t}dt<\infty. 
\end{align*}

We will now establish a bound on $f'_{h_{k+1}}$. In fact since 
\begin{align*}
 f'_{h_{k+1}}(y)&=f_{h_{k+1}}(y)-\alpha y^{-\alpha-1}e^y\int_0^y\brac{h_{k+1}(t)-\E{h_{k+1}(\mathcal{G}_\alpha)}}t^{\alpha-1}e^{-t}dt\\
 &\quad +\frac{h_{k+1}(y)-\E{h_{k+1}(\mathcal{G}_\alpha)}}{y}, 
\end{align*}
the argument to bound $f'_{h_{k+1}}$ will also yield a bound on $f_{h_{k+1}}$. Therefore a bit unconventionally, we will proceed to bound the derivative first and consider the original function later.

We have, for $y \neq 0$,
\begin{align}
\label{formula_derivativef_hx>0}
    f'_{h_{k+1}}(y)&=\brac{y^{-\alpha}-\alpha y^{-\alpha-1}}e^y\int_0^y\brac{h_{k+1}(t)-\E{h_{k+1}(\mathcal{G}_\alpha)}}t^{\alpha-1}e^{-t}dt\nonumber\\
    &\quad +\frac{h_{k+1}(y)-\E{h_{k+1}(\mathcal{G}_\alpha)}}{y}.
\end{align}
In the case where $-\infty<y<0$, we can write
\begin{align}
\label{estimate_fh'y<0}
\abs{f'_{h_{k+1}}(y)}&\leq \abs{\E{h_{k+1}(\mathcal{G}_\alpha)}}\brac{\abs{y}^{-\alpha}+\alpha \abs{y}^{-\alpha-1}}e^y\int_y^0 \abs{t}^{\alpha-1}e^{-y}dt+\frac{\abs{\E{h_{k+1}(\mathcal{G}_\alpha)}}}{\abs{y}}\nonumber\\
&\leq \frac{\abs{\E{h_{k+1}(\mathcal{G}_\alpha)}}}{\alpha}\brac{1+\frac{\alpha} {\abs{y}}}+\frac{\abs{\E{h_{k+1}(\mathcal{G}_\alpha)}}}{\abs{y}}, 
\end{align}
noting that
$\E{h_{k+1}(\mathcal{G}_\alpha)}=p^{(k)}_{\mathcal{G}_\alpha}(x)$ by
Corollary \ref{corollary_densitygamma}. In the case where $0< y\leq
x$, we have
\begin{align}
\label{estimate_fh')<y<x}
    \abs{f'_{h_{k+1}}(y)}&\leq \abs{\E{h_{k+1}(\mathcal{G}_\alpha)}}\brac{y^{-\alpha}+\alpha y^{-\alpha-1}}e^x\int_0^y t^{\alpha-1}dt+\frac{\abs{\E{h_{k+1}(\mathcal{G}_\alpha)}}}{y}\nonumber\\
&\leq \frac{\abs{\E{h_{k+1}(\mathcal{G}_\alpha)}}}{\alpha}e^x\brac{1+\frac{\alpha} {y}}+\frac{\abs{\E{h_{k+1}(\mathcal{G}_\alpha)}}}{y}.
\end{align}
The case where $y>x$ requires more efforts. The Gamma density
representation in Corollary \ref{corollary_densitygamma} implies that
\begin{align*}
    \int_0^y h_{k+1}(t)t^{\alpha-1}e^{-t}dt=\Gamma(\alpha)\brac{p^{(k)}_{\mathcal{G}_\alpha}(x)-p^{(k)}_{\mathcal{G}_\alpha}(y)}.
\end{align*}
We also know that
\begin{align*}
\int_0^y t^{\alpha-1}e^{-t}dt=\Gamma(\alpha)-\int_y^\infty t^{\alpha-1}e^{-t}dt. 
\end{align*}
We hence deduce from \eqref{formula_derivativef_hx>0} that when $y>x$,
we have
\begin{align}
\label{equation_derivativefh_y>x>0}
    f'_{h_{k+1}}(y)
    &=-\Gamma(\alpha)\brac{y^{-\alpha}-\alpha y^{-\alpha-1}}e^yp^{(k)}_{\mathcal{G}_\alpha}(y)\nonumber\\
    &\quad +p^{(k)}_{\mathcal{G}_\alpha}(x)\brac{y^{-\alpha}-\alpha y^{-\alpha-1}}e^y\int_y^\infty t^{\alpha-1}e^{-t}dt+\frac{h_{k+1}(y)-\E{h_{k+1}(\mathcal{G}_\alpha)}}{y}\nonumber\\
   &=\mathcal{A}_1+\mathcal{A}_2+ \frac{h_{k+1}(y)-\E{h_{k+1}(\mathcal{G}_\alpha)}}{y}. 
\end{align}
Let us consider each individual term on the right hand side of \eqref{equation_derivativefh_y>x>0}. Regarding $\mathcal{A}_1$, identity \eqref{formula_gammadensityderivative_original} implies the bound
\begin{align*}
   \mathcal{A}_1&= -\Gamma(\alpha)\brac{y^{-\alpha}-\alpha y^{-\alpha-1}}e^yp^{(k)}_{\mathcal{G}_\alpha}(y)\\
    &\leq \brac{y^{-1}+\alpha y^{-2}}\sum_{i=0}^{k}{k\choose i}\prod_{j=1}^i\abs{j-\alpha}\frac{1}{y^i}\nonumber\\
    &\leq \brac{\frac{1}{x}+\alpha \frac{1}{x^2}}\sum_{i=0}^{k}{k\choose i}\prod_{j=1}^i\abs{j-\alpha}\frac{1}{x^i}. 
\end{align*}
Next, we handle the term $\mathcal{A}_2$. Let $N=\ceil{\alpha}+1$.  By performing integration by parts iteratively, one obtains the identity
\begin{align*}
    \int t^{\alpha-1}e^{-t}dt=-\sum_{n=1}^{N-1} \prod_{k=1}^{n-1}(\alpha-k) t^{\alpha-n}e^{-t}+\prod_{k=1}^{N-1}(\alpha-k)\int t^{\alpha-N}e^{-t}dt.
\end{align*}
With the convention that $\prod_{k=1}^0(\alpha-k)=1$, we can write
\begin{align*}
    e^y\int_y^\infty t^{\alpha-1}e^{-t}dt=
    \sum_{n=1}^{N-1} \prod_{k=1}^{n-1}(\alpha-k)y^{\alpha-n}+\prod_{k=1}^{N-1}(\alpha-k)e^y\int_y^\infty t^{\alpha-N}e^{-t}dt. 
\end{align*}
In particular, since $\alpha-N+1<0$, we can bound the last integral on
the right hand side of the above identity by
\begin{align*}
    e^y\int_y^\infty t^{\alpha-N}e^{-t}dt\leq  e^y\int_y^\infty t^{\alpha-N}e^{-y}dt=\frac{y^{\alpha-N+1}}{N-\alpha-1}. 
\end{align*}
This then yields
\begin{align*}
   \mathcal{A}_2&= p^{(k)}_{\mathcal{G}_\alpha}(x)\brac{y^{-\alpha}-\alpha y^{-\alpha-1}}e^y\int_y^\infty t^{\alpha-1}e^{-t}dt\\
    &\leq  p^{(k)}_{\mathcal{G}_\alpha}(x)\brac{y^{-\alpha}+\alpha y^{-\alpha-1}}\\
    &\qquad\times \brac{ \sum_{n=1}^{N-2} \prod_{k=1}^{n-1}(\alpha-k)y^{\alpha-n}+\prod_{k=1}^{N-2}(\alpha-k)y^{\alpha-N+1}+\prod_{k=1}^{N-1}(\alpha-k)\frac{y^{\alpha-N+1}}{N-\alpha-1}}\\
    &\leq  p^{(k)}_{\mathcal{G}_\alpha}(x)\brac{x^{-\alpha}+\alpha x^{-\alpha-1}}\\
    &\qquad\times\brac{\sum_{n=1}^{\ceil{\alpha}-1} \prod_{k=1}^{n-1}(\alpha-k)y^{\alpha-n}+\prod_{k=1}^{\ceil{\alpha}-1}(\alpha-k)x^{\alpha-\ceil{\alpha}}+\prod_{k=1}^{\ceil{\alpha}}(\alpha-k)\frac{x^{\alpha-\ceil{\alpha}}}{\ceil{\alpha}-\alpha}},
\end{align*}
where the last inequality comes from the fact that $y>x>0$ and
$N=\ceil{\alpha}+1$. Finally, we deduce from
\eqref{equation_derivativefh_y>x>0} that when $y>x$, it holds that
\begin{align}
\label{estimate_fh'y>x}
    \abs{f'_{h_{k+1}}(y)}&\leq \brac{\frac{1}{x}+\alpha \frac{1}{x^2}}\sum_{i=0}^{k}{k\choose i}\prod_{j=1}^i\abs{j-\alpha}\frac{1}{x^i}
    +p^{(k)}_{\mathcal{G}_\alpha}(x)\brac{x^{-\alpha}+\alpha x^{-\alpha-1}}\nonumber\\
    &\qquad\times \brac{\sum_{n=1}^{\ceil{\alpha}-1} \prod_{k=1}^{n-1}(\alpha-k)y^{\alpha-n}+\prod_{k=1}^{\ceil{\alpha}-1}(\alpha-k)x^{\alpha-\ceil{\alpha}}+\prod_{k=1}^{\ceil{\alpha}}(\alpha-k)\frac{x^{\alpha-\ceil{\alpha}}}{\ceil{\alpha}-\alpha}}\nonumber\\
    &\quad +\sum_{i=0}^{k+1}{k+1\choose i}\prod_{j=1}^i\abs{j-\alpha}\frac{1}{x^{i+1}}+\frac{\abs{\E{h_{k+1}(\mathcal{G}_\alpha)}}}{x}. 
\end{align}
Combining all of the above, the estimate on
$\abs{f'_{h_{k+1}}}$ at \eqref{estimate_derivativesolutionstein_x>0} follows from
\eqref{estimate_fh'y<0}, \eqref{estimate_fh')<y<x}, \eqref{estimate_fh'y>x} and Minkowski's inequality.

Now let us consider $f_{h_{k+1}}$. In a similar fashion, the estimates at \eqref{estimate_fh'y<0}, \eqref{estimate_fh')<y<x} and \eqref{estimate_fh'y>x} imply for $y\in \R$, 
\begin{align*}
   \abs{f_{h_{k+1}}(y)}&\leq     \frac{\abs{\E{h_{k+1}(\mathcal{G}_\alpha)}}}{\alpha}+\frac{\abs{\E{h_{k+1}(\mathcal{G}_\alpha)}}}{\alpha}e^x\\
   &+\frac{1}{x}\sum_{i=0}^{k}{k\choose i}\prod_{j=1}^i\abs{j-\alpha}\frac{1}{x^i}
    +p^{(k)}_{\mathcal{G}_\alpha}(x)x^{-\alpha}\nonumber\\
    &\quad \times\brac{\sum_{n=1}^{\ceil{\alpha}-1} \prod_{k=1}^{n-1}(\alpha-k)y^{\alpha-n}+\prod_{k=1}^{\ceil{\alpha}-1}(\alpha-k)x^{\alpha-\ceil{\alpha}}+\prod_{k=1}^{\ceil{\alpha}}(\alpha-k)\frac{x^{\alpha-\ceil{\alpha}}} {\ceil{\alpha}-\alpha}}.
\end{align*}

It remains to prove estimate \eqref{esimate_steinmethod_x>0}. To this
end, we refer to \cite[Theorem 2.9.1]{nourdinpeccatibook}: it holds for real-valued function $\psi$ of class $C^1$ with bounded derivative that 
\begin{align}
\label{intbypartnourdinpeccatibook}
    \E{F\psi(F+\alpha)}=\E{\psi'(F+\alpha)\inner{DF,-DL^{-1}F}_\frak{H}}. 
\end{align}
Now we claim that the assumptions $\E{\frac{1}{(F+\alpha)^2}}<\infty$, $\E{(F+\alpha)^{2(\alpha-i)}}<\infty$ for $1\leq i\leq \ceil{\alpha}-1$ and the newly obtained bound on $f_{h_{k+1}}, f'_{h_{k+1}}$ guarantee 
\begin{align}
\label{requiredbound}
    \E{\abs{Ff_{h_{k+1}}(F+\alpha)}}<\infty;\qquad \E{\abs{f'_{h_{k+1}}(F+\alpha)\inner{DF,-DL^{-1}F}_\frak{H}}}<\infty. 
\end{align}

We shall prove the above second bound and the first is easier. In fact, we have
\begin{align*}
    \left|\E{\abs{f'_{h_{k+1}}(F+\alpha)\inner{DF,-DL^{-1}F}_\frak{H}}} \right|^2 \leq \E{\abs{f'_{h_{k+1}}(F+\alpha)}^2}\times \E{\abs{\inner{DF,-DL^{-1}F}_\frak{H}}^2}. 
\end{align*}
The assumption $F\in \mathbb{D}^{1,4}$ guarantees that $\E{\abs{\inner{DF,-DL^{-1}F}_\frak{H}}^2}<\infty$. Moreover, \eqref{estimate_derivativesolutionstein_x>0} and Jensen's inequality imply that
\begin{align*}
    &\E{\abs{f'_{h_{k+1}}(F+\alpha)}^2}\\
    &\qquad\qquad\leq (\ceil{\alpha}+1)\brac{d_1(x)^2\sum_{i=1}^{\ceil{\alpha}-1} \E{\abs{F+\alpha}^{2(\alpha-i)}}+d_2(x)^2\E{\abs{\frac{1}{F+\alpha}}^2}+d_3(x)^2 }, 
\end{align*}
at which point the assumptions $\E{\abs{\frac{1}{F+\alpha}}^2}<\infty$, $\E{\abs{F+\alpha}^{2(\alpha-i)}}<\infty$ for $1\leq i\leq \ceil{\alpha}-1$ ensure the right hand side is finite for every $x>0$. We can conclude from these estimates that $ \E{\abs{f'_{h_{k+1}}(F+\alpha)\inner{DF,-DL^{-1}F}_\frak{H}}}<\infty$.

Thus, under \eqref{estimate_derivativesolutionstein_x>0}, the dominated convergence theorem and \eqref{requiredbound} imply \eqref{intbypartnourdinpeccatibook} holds for $f_{h_{k+1}}$, that is
\begin{align}
\label{intbypartforunboundedfunction}
    \E{Ff_{h_{k+1}}(F+\alpha)}=\E{f'_{h_{k+1}}(F+\alpha)\inner{DF,-DL^{-1}F}_\frak{H}}. 
\end{align}
This leads to
\begin{align*}
    \abs{\E{h_{k+1}(F+\alpha) }-\E{h_{k+1}(\mathcal{G}_\alpha) }}&=\abs{\E{(F+\alpha)f'_{h_{k+1}}(F+\alpha) }-\E{Ff_{h_{k+1}}(F+\alpha)}}\\
                                                                 &\leq  \E{\brac{f'_{h_{k+1}}(F+\alpha)}^2}^{1/2}\\
  &\hspace{7em}\times\E{\brac{{{F+\alpha}} -\inner{DF,-DL^{-1}F}_\mathfrak{H}}^2}^{1/2},
\end{align*}
at which point we apply estimate
\eqref{estimate_derivativesolutionstein_x>0}, concluding the proof.
\end{proof}
\begin{remark} It gives the impression that     $F$ 
needs to be more regular than in  $\mathbb{D}^{1,4}$ 
since the appearance of $L$ in \eqref{esimate_steinmethod_x>0}.  However, by the 
multiplier theorem (\cite[Corollary 7.1]{hubook}),
$L^{-1}$ is a bounded operator in     $L^p$ 
  for any $p>1$.
\end{remark} 



The following proposition is the analog of Proposition \ref{prop_steinmethod_x>0} for
another but related class of test functions that is suitable for approximating $p_{\mathcal{G}_\alpha}(x)$ for $x<0$. 
\begin{proposition}
\label{prop_steinmethod_x<0}
Let $x<0$ and the test function 
in the ordinary differential equation stated at \eqref{equation_gammastein} be
\begin{align*}
    h_{k+1}(y)=\mathds{1}_{\{y\leq x\}}\nu_{k+1}(y), \quad y\in \R
\end{align*}
where $\nu_{k+1},k\geq 0$, is defined at \eqref{def_nuk}. Then, a
solution $f_{h_{k+1}}:\R\setminus \{0\}\to \R$ to equation \eqref{equation_gammastein} exists and satisfies, for every $y\neq 0$,
\begin{align}
\label{estimate_derivativesolutionstein_x<0}
   \max\left\{ \abs{f_{h_{k+1}}(y)},  \abs{f'_{h_{k+1}}(y)}\right\}\leq   e_1(x)\sum_{i=0}^{\ceil{\alpha}-1}\abs{y}^{\alpha-i}+e_2(x),
\end{align}
where the factors $e_1(x)$ and $e_2(x)$ are positive  and finite for every $x<0$. Consequently, if $F$ is a random variable in $\mathbb{D}^{1,4}$ such that $\E{F}=0$ and $\E{F^2}=\alpha$, then, for
every $x<0$, we have the pointwise estimate
    \begin{align}
    \label{esimate_steinmethod_x<0}
      &\abs{\E{\mathds{1}_{\{F+\alpha\leq x\}}\nu_{k+1}(F+\alpha)}- \E{\mathds{1}_{\{\mathcal{G}_{\alpha}\leq x\}}\nu_{k+1}(\mathcal{G}_{\alpha})}}\nonumber\\
      &\qquad\qquad\qquad \leq  \brac{e_1(x)\sum_{i=0}^{\ceil{\alpha}-1}\E{\brac{F+\alpha}^{2(\alpha-i)}}^{1/2}+e_2(x)}\nonumber\\
      &\hspace{17em}\times \E{\brac{{{F+\alpha}} -\inner{DF,-DL^{-1}F}_\mathfrak{H}}^2}^{1/2}.
    \end{align}
\end{proposition}
\begin{proof}
The fact that $\mathcal{G}_\alpha>0$ almost surely implies that
$\E{h_{k+1}(\mathcal{G}_\alpha)}=0$ when $x<0$. A solution to the
Stein equation \eqref{equation_gammastein} is then
\begin{align*}
    f_{h_{k+1}}(y)=e^yy^{-\alpha}\int_0^y{h_{k+1}(t)}t^{\alpha-1}e^{-t}dt. 
\end{align*}
Bounding $f_{h_{k+1}}$ is similar but simpler than bounding $f'_{h_{k+1}}$, so we will only show the latter. When $y\geq x$ and $y\neq 0$, the definition of $h_{k+1}$ states that
$f_{h_{k+1}}(y)=0$ and therefore $f'_{h_{k+1}}(y)=0$ on
$[x,\infty)$. What remains is the case where $y<x<0$.  We have
\begin{align*}
    \abs{f'_{h_{k+1}}(y)}&=\abs{\brac{{y}^{-\alpha}+\alpha {y}^{-\alpha-1}}e^y\int_x^y{{\nu_{k+1}(t)}}{t}^{\alpha-1}e^{-t}dt+{\frac{h_{k+1}(y)}{y}}}\\
    &\leq \brac{\abs{y}^{-\alpha}+\alpha \abs{y}^{-\alpha-1}}e^y\int_y^x{\abs{\nu_{k+1}(t)}}\abs{t}^{\alpha-1}e^{-y}dt+\abs{\frac{h_{k+1}(y)}{y}}\\
    &\leq \brac{\abs{y}^{-\alpha}+\alpha \abs{y}^{-\alpha-1}}\sum_{i=0}^{k+1}e_i\brac{\abs{x}^{\alpha-i}+\abs{y}^{\alpha-i}}+\sum_{i=0}^{k+1}\frac{f_i}{\abs{y}^{i+1}}
\end{align*}
where $e_i,f_i$, $0 \leq i \leq k+1$, are some positive constants. Since $y<x<0$, it follows that
\begin{align*}
    &\abs{f'_{h_{k+1}}(y)}\\&\leq  \brac{\abs{y}^{-\alpha}+\alpha \abs{y}^{-\alpha-1}}\brac{\sum_{i=0}^{k+1}e_i\abs{x}^{\alpha-i}+\sum_{i=0}^{\ceil{\alpha}-1}e_i\abs{y}^{\alpha-i}+\sum_{i=\ceil{\alpha}}^{k+1}e_i\abs{y}^{\alpha-i}}
    +\sum_{i=0}^{k+1}\frac{f_i}{\abs{y}^{i+1}}\\
    &\leq  \brac{\abs{x}^{-\alpha}+\alpha \abs{x}^{-\alpha-1}}\brac{\sum_{i=0}^{k+1}e_i\abs{x}^{\alpha-i}+\sum_{i=0}^{\ceil{\alpha}-1}e_i\abs{y}^{\alpha-i}+\sum_{i=\ceil{\alpha}}^{k+1}e_i\abs{x}^{\alpha-i}}
    +\sum_{i=0}^{k+1}\frac{f_i}{\abs{x}^{i+1}}. 
\end{align*}
The above inequalities and Minkowski's inequality then imply the
estimate \eqref{estimate_derivativesolutionstein_x<0}. Next, the
estimate \eqref{esimate_steinmethod_x<0} is obtained in the same way
as the estimate \eqref{esimate_steinmethod_x>0}. 
\end{proof}

Finally, we state and prove a third analogous proposition to Propositions
\ref{prop_steinmethod_x>0} and \ref{prop_steinmethod_x<0} for a third kind of test functions that is suitable for approximating $p_{\mathcal{G}_\alpha}(0)$.

\begin{proposition}
  \label{prop_steinmethod_x=0}
Let $\alpha>k+1$ and the test function $h$ in the ordinary differential equation stated at \eqref{equation_gammastein} be
\begin{align*}
    m_{k+1}(y)=\mathds{1}_{\{y< 0\}}\nu_{k+1}(y), \quad y\in \R
\end{align*}
where $\nu_{k+1},k\geq 0$, is defined at \eqref{def_nuk}. Then, a
solution $f_{m_{k+1}}:\R\setminus \{ 0\}\to \R$ to equation \eqref{equation_gammastein} exists
and satisfies, for every $y\neq 0$,
\begin{align*}
    \max\left\{ \abs{g_{m_{k+1}}(y)},  \abs{g'_{m_{k+1}}(y)}\right\}\leq   f_1\sum_{i=0}^{\ceil{\alpha}-1}\abs{y}^{\alpha-i}+f_2,
\end{align*}
where the quantities $f_1$ and $f_2$ are positive
constants. Consequently, if $F$ is a random variable in
$\mathbb{D}^{1,4}$ such that $\E{F}=0$ and $\E{F^2}=\alpha$, then, we
have the estimate
    \begin{align*}
      &\abs{\E{\mathds{1}_{\{F+\alpha< 0\}}\nu_{k+1}(F+\alpha)}- \E{\mathds{1}_{\{\mathcal{G}_{\alpha}< 0\}}\nu_{k+1}(\mathcal{G}_{\alpha})}} \\&\leq  \brac{f_1\sum_{i=0}^{\ceil{\alpha}-1}\E{\brac{F+\alpha}^{2(\alpha-i)}}^{1/2}+f_2}
      \E{\brac{{F+\alpha} -\inner{DF,-DL^{-1}F}}^2}^{1/2}.
    \end{align*}
\end{proposition}
\begin{proof}
 The function
    \begin{align*}
    f_{m_{k+1}}(y)=e^yy^{-\alpha}\int_0^y{m_{k+1}(t)}t^{\alpha-1}e^{-t}dt
\end{align*}
is a solution to our Stein's equation. Due to the definition of $\nu_{k+1}$ at \eqref{def_nuk}, this solution contains the integral $\int_0^y t^{\alpha-k-2}e^{-t}dt$ which
converges when $\alpha-k-2>-1$ and diverges otherwise. Hence,
$f_{m_{k+1}}$ is well-defined on $\R\setminus \{ 0\}$ only when $\alpha>k+1$. The rest of the
proof is the same as that of Proposition
\ref{prop_steinmethod_x<0}.
\end{proof}


\section{Random variables in a fixed Wiener chaos of even order}
\label{sectionfixedchaos}

In this section, we assume that $F$ is a multiple Wiener integral of an
even order $q \geq 2$. We will prove our main results which are Theorem \ref{theorem_fourthmoment_singlechaos_0thderivative} and Theorem \ref{theorem_fourthmoment_singlechaos_kthderivative}. The reason we consider even $q\geq 2$ is explained in the introduction, Remark \ref{remark_eveninteger} and Remark \ref{remark_eveninteger_secondplace}.

\subsection{Pointwise estimates for densities}
~\

We begin by stating and proving preliminary results which will
constitute the building blocks of the proof of the Malliavin derivative estimate in Proposition \ref{prop_densityestimate_DF^2_0thderivative}.

The following technical lemma which expands $\E{\brac{w-q(F+\alpha) }^2}$ into sum of contraction norms is well known in the Malliavin-Stein's community, see for instance \cite[Proposition 3.13]{NP09main}. We include its short proof here for the sake of completeness. 
\begin{lemma}
\label{lemma_DF^2gammaconvergencecondition}
Let $q \geq 2$ be an even integer and $f \in
\mathfrak{H}^{\otimes q}$. Define $F=I_q(f)$, then it holds that
\begin{align*}
    \E{\brac{w-q(F+\alpha) }^2}&=q^4\sum_{\substack{r=0\\r\neq q/2-1}}^{q-2} r!^2 {q-1\choose r}^4(2q-2-2r)!\norm{f\widetilde{\otimes}_{r+1} f}^2_{\mathfrak{H}^{ \otimes 2q-2r-2}}\\
    &\quad +(q-1)!q^3\norm{q(q/2-1)! {q-1\choose q/2-1}^2f\widetilde{\otimes}_{q/2} f-f }_{\mathfrak{H}^{\otimes q}}^2.
\end{align*}
\end{lemma}
\begin{proof}
The fact that $DF=qI_{q-1}(f)$ together with the product formula for
multiple Wiener integrals \eqref{prodformula} imply
\begin{align*}
    \E{w^2}&=q^4\sum_{r=0}^{q-1} r!^2 {q-1\choose r}^4(2q-2-2r)!\norm{f\widetilde{\otimes}_{r+1} f}^2_{\mathfrak{H}^{ \otimes 2q-2r-2}}.
\end{align*}
Since $\alpha=\E{F^2}=q!\norm{f}^2_{\mathfrak{H}^{\otimes q}}$, it follows that the above term corresponding to $r=q-1$ is $q^4(q-1)!^2\norm{f}^4_{\mathfrak{H}^{\otimes q}}=q^2\alpha^2$. Then the above equation becomes
\begin{align*}
    \E{w^2}&=q^4\sum_{\substack{r=0\\r\neq q/2-1}}^{q-2} r!^2 {q-1\choose r}^4(2q-2-2r)!\norm{f\widetilde{\otimes}_{r+1} f}^2_{\mathfrak{H}^{ \otimes 2q-2r-2}}\\
    &\quad +q^2\alpha^2+(q-1)!q^3\brac{q(q/2-1)! {q-1\choose q/2-1}^2\norm{f\widetilde{\otimes}_{q/2} f}_{\mathfrak{H}^{\otimes q}}}^2.
\end{align*}
Furthermore, the same product formula allows one to prove that
\begin{align*}
    2q\E{Fw}=2(q-1)!q^3 \brac{q(q/2-1)! {q-1\choose q/2-1}^2\inner{f\widetilde{\otimes}_{q/2} f,f}_{\mathfrak{H}^{\otimes q}}}.
\end{align*}
Moreover, it is relatively straightforward to see $q^2\E{(F+\alpha)^2}=q^2\alpha^2+(q-1)!q^3\norm{f}^2_{\mathfrak{H}^{\otimes
    q}}$ and $-2q \alpha\E{w}=-2q^2\alpha^2$, so that combining the previous identities yields
\begin{align*}
    \E{\brac{w-q(F+\alpha) }^2}&=q^4\sum_{\substack{r=0\\r\neq q/2-1}}^{q-2} r!^2 {q-1\choose r}^4(2q-2-2r)!\norm{f\widetilde{\otimes}_{r+1} f}^2_{\mathfrak{H}^{ \otimes 2q-2r-2}}\\
    &\quad +(q-1)!q^3\norm{q(q/2-1)! {q-1\choose q/2-1}^2f\widetilde{\otimes}_{q/2} f-f }_{\mathfrak{H}^{\otimes q}}^2,
\end{align*}
which concludes the proof.
\end{proof}

The following lemma can be viewed as a higher order (derivative) version of the previous lemma.  
\begin{lemma}
\label{lemma_compute_D2F}
Let $q \geq 2$ be an even integer and $f \in
\mathfrak{H}^{\otimes q}$. Define $F=I_q(f)$, then it holds that
\begin{align*}
    &\E{\norm{2D^2F\otimes_1 DF-q DF}^2_\mathfrak{H}}\\
&\qquad =4q^4(q-1)^2\sum_{\substack{r=0\\r\neq q/2-1}}^{q-2}r!^2\brac{\frac{q-r-1}{q-1}}^2{q-1\choose r}^4(2q-3-2r)!\norm{f\widetilde{\otimes}_{r+1}f}^2_{\mathfrak{H}^{\otimes2q-2-2r}}\\
&\qquad\quad +(q-1)!q^4 \norm{  q (q/2-1)!{q-1\choose q/2-1}^2 f\widetilde{\otimes}_{q/2} f- f }_{\mathfrak{H}^{\otimes q}}^2.
\end{align*}
As a consequence, there exists a constant $C_1>0$ depending on $q$ such that
\begin{align}
\label{estimate_D^2F-qDF}
\E{\norm{\Lambda(2,1)}_\mathfrak{H}^2}&=  \E{\norm{\frac{2}{q}D^2F\otimes_1 DF- DF}^2_\mathfrak{H}}\nonumber\\
&\leq  C_1\E{\brac{{q\brac{F+\alpha}} -\norm{DF}^2_\mathfrak{H}}^2},
\end{align}
where $\Lambda(2,1)$ is the quantity defined by  \eqref{def_Lambda}.
\end{lemma}

\begin{remark}
\label{remark_eveninteger}
  We do not know if \eqref{estimate_D^2F-qDF} holds true for   any positive integer or it holds true  only for even integer greater or equal to $2. $
\end{remark}

\begin{proof}
As $D^2 F=q(q-1)I_{q-2}(f)$, the product formula for multiple Wiener
integrals \eqref{prodformula} yields
\begin{align*}
    D^2F\otimes_1 DF=q^2(q-1)\sum_{r=0}^{q-2}r!{q-2\choose r}{q-1\choose r}I_{2q-3-2r}\brac{f\widetilde{\otimes}_{r+1}f},
\end{align*} 
which implies
\begin{align*}
    2\E{\inner{2 D^2F\otimes_1 DF,q
  DF}_\mathfrak{H}}&=4q^4(q-1)(q-1)!(q/2-1)!{q-2\choose
                     q/2-1}{q-1\choose q/2-1}\\
  &\qquad\qquad\qquad\qquad\qquad\qquad\qquad\qquad \times\inner{ f\widetilde{\otimes}_{q/2}f , f}_{\mathfrak{H}^{\otimes q}}.
\end{align*}
Similarly, we get
\begin{align*}
  &4\E{\norm{D^2F\otimes_1 DF}^2_\mathfrak{H}}\\
  &\qquad =4q^4(q-1)^2\sum_{\substack{r=0\\r\neq q/2-1}}^{q-2}r!^2{q-2\choose r}^2{q-1\choose r}^2(2q-3-2r)!\norm{f\widetilde{\otimes}_{r+1}f}^2_{\mathfrak{H}^{\otimes2q-2-2r}}\\
    &\qquad\quad +(q-1)!\brac{2q^2(q-1)(q/2-1)!{q-2\choose q/2-1}{q-1\choose q/2-1}\norm{f\widetilde{\otimes}_{q/2} f}_{\mathfrak{H}^{\otimes q}}}^2
\end{align*}
and $q^2\E{w}=(q-1)!\brac{q^2\norm{f}_{\mathfrak{H}^{\otimes q}}}^2$. Combining the previous identities leads to
\begin{align*}
&\E{\norm{2D^2F\otimes_1 DF-q DF}^2_\mathfrak{H}}\\
& =4q^4(q-1)^2\sum_{\substack{r=0\\r\neq q/2-1}}^{q-2}r!^2{q-2\choose r}^2{q-1\choose r}^2(2q-3-2r)!\norm{f\widetilde{\otimes}_{r+1}f}^2_{\mathfrak{H}^{\otimes2q-2-2r}}\\
&\qquad +(q-1)!\norm{ 2q^2(q-1)(q/2-1)!{q-2\choose q/2-1}{q-1\choose q/2-1}f\widetilde{\otimes}_{q/2} f-q^2 f }_{\mathfrak{H}^{\otimes q}}^2\\
& =4q^4(q-1)^2\sum_{\substack{r=0\\r\neq q/2-1}}^{q-2}r!^2\brac{\frac{q-r-1}{q-1}}^2{q-1\choose r}^4(2q-3-2r)!\norm{f\widetilde{\otimes}_{r+1}f}^2_{\mathfrak{H}^{\otimes2q-2-2r}}\\
&\quad\quad +(q-1)!q^4 \norm{  q (q/2-1)!{q-1\choose q/2-1}^2 f\widetilde{\otimes}_{q/2} f- f }_{\mathfrak{H}^{\otimes q}}^2,
\end{align*}
where the last equality is a consequence of ${q-2\choose
  r}=\frac{q-r-1}{q-1}{q-1\choose r}$, which implies that
$\brac{{2}/{q}}{q-2\choose q/2-1}(q-1)={q-1 \choose q/2-1}$. In view
of Lemma \ref{lemma_DF^2gammaconvergencecondition}, the estimate \eqref{estimate_D^2F-qDF} immediately follows. 
\end{proof}


In the upcoming result, thank to the previous lemmas, we are able to bound the difference of the density of $F+\alpha$ and the density of $\mathcal{G}_\alpha$ by a quantity which contains the norm of the Malliavin derivative of $F$. 

\begin{proposition}
\label{prop_densityestimate_DF^2_0thderivative}
Let $F$ be a multiple Wiener integral of an even order $q\geq 2$ such
that $\E{F^2}=\alpha$. Assume that the quantities
$\E{(F+\alpha)^{-2}}$, $\E{w^{-2}}$, $\E{(F+\alpha)^{-2}w^{-2} }$,
$\E{w^{-3}}$ and $\E{(F+\alpha)^{2(\alpha-i)}}$, $1\leq i\leq
\ceil{\alpha}-1 $, exist and are finite. Then, for every $x\neq 0$ and
$\alpha>0$, we have the pointwise density estimate 
\begin{align*}
    & \abs{p_{F+\alpha}(x)-p_{\mathcal{G}_\alpha}(x)}\\
    &\leq \Bigg(d_1(x)\sum_{i=1}^{\ceil{\alpha}-1}\E{(F+\alpha)^{2(\alpha-i)}}^{1/2}+ {d_2(x)}\E{\frac{1}{(F+\alpha)^2}}^{1/2}+{d_3(x)}\\
    &\quad
      +\E{\frac{1}{w^2}}^{1/2}+\abs{\alpha-1}\E{\frac{1}{w^2(F+\alpha)^2}}^{1/2}
      +C_1\E{\frac{1}{w^3}}^{1/2}\Bigg)\\
  &\qquad\qquad\qquad\qquad\qquad\qquad\qquad\qquad\qquad  \times \E{\brac{{q\brac{F+\alpha}} -w}^2}^{1/2},
\end{align*}
where the factors $d_1(x),d_2(x),d_3(x)$ are positive and finite for
every $x \neq 0$, and $C_1$ is some positive constant that is independent of $x$ and appears first in Lemma \ref{lemma_compute_D2F}. Moreover, for $x=0$, the same estimate holds if $\alpha>1$. 
\end{proposition}
\begin{proof}
Let us start with the case where $x>0$. Using the density
representation formula from Lemma \ref{lemma_decompositiondensityforgammaconvergence}, we can write
\begin{align*}
    &p_{F+\alpha}(x)=\E{\mathds{1}_{\{F+\alpha>x\}}\delta\brac{u}}=\E{\mathds{1}_{\{F+\alpha>x\}}\brac{\nu_1(F+\alpha)+T_1} },
\end{align*}
where $\nu_1(y)=1+\frac{1-\alpha}{y}$. Recall that the quantity $T_1$ was computed
in \eqref{equation_T1} and is given by 
\begin{align*}
    T_1=\brac{-\frac{1}{w}-\frac{(1-\alpha)}{w(F+\alpha)} }\brac{w- q\brac{F+\alpha}} +\frac{1}{w^2}\inner{\frac{2}{q}D^2F\otimes_1 DF-DF,DF}_\mathfrak{H}. 
\end{align*}
By using the Gamma density representation from Corollary \ref{corollary_densitygamma}, we have
\begin{align*}
    &\abs{p_{F+\alpha}(x)-p_{\mathcal{G}_\alpha}(x)}\leq \abs{\E{\mathds{1}_{\{F+\alpha>x\}}\nu_1(F+\alpha)}-\E{\mathds{1}_{\{\mathcal{G}_{\alpha}>x\}}\nu_1(\mathcal{G}_\alpha)}}+\E{\abs{T_1}}. 
\end{align*}
The first term on the right hand side can be bounded by using the
estimate \eqref{esimate_steinmethod_x>0} in Proposition
\ref{prop_steinmethod_x>0}, so that it remains to deal with
$\E{\abs{T_1}}$. To that end, we observe that
 \begin{align*}
     &\E{\abs{\frac{1}{w^2}\inner{\frac{2}{q}D^2F\otimes_1 DF-DF,DF}_\mathfrak{H}}}\\
     &=\E{\abs{\frac{1}{w^{3/2}}\inner{\frac{2}{q}D^2F\otimes_1 DF-DF,\frac{DF}{\norm{DF}_\mathfrak{H}}}_\mathfrak{H}}}\nonumber\\
     &\leq\E{\frac{1}{w^3}}^{1/2}\E{\norm{\frac{2}{q}D^2F\otimes_1 DF- DF}^2_\mathfrak{H}}^{1/2}\nonumber\\
     &\leq C_1\E{\frac{1}{w^3}}^{1/2}\E{\brac{{q\brac{F+\alpha}} -\norm{DF}^2_\mathfrak{H}}^2},
 \end{align*}
where the last line is a consequence of estimate
\eqref{estimate_D^2F-qDF} in Lemma \ref{lemma_compute_D2F}. In particular, we use the fact that $q\geq 2$ is an even integer here.

The above inequality leads to
\begin{align*}
    \E{\abs{T_1}}&\leq \brac{\E{\frac{1}{w^2}}^{1/2}+\abs{\alpha-1}\E{\frac{1}{w^2(F+\alpha)^2}}^{1/2}}\E{\brac{w- q\brac{F+\alpha}}^2}^{1/2}\\
    &\quad +C_1\E{\frac{1}{w^3}}^{1/2}\E{\brac{w-{q\brac{F+\alpha}}}^2},
\end{align*}
which yields the desired conclusion in the case $x>0$. We now address
the case where $x\leq 0$. Lemma
\ref{lemma_decompositiondensityforgammaconvergence} allows us to write
\begin{align*}
    p_{F+\alpha}(x)=\E{\mathds{1}_{\{F+\alpha>x\}}\delta\brac{u}}&=\E{\brac{1-\mathds{1}_{\{F+\alpha\leq x\}}}\delta\brac{u}}\\
    &=-\E{\mathds{1}_{\{F+\alpha\leq x\}}\delta\brac{u}}\\
    & =-\E{\mathds{1}_{\{F+\alpha\leq x\}}\brac{\nu_1(F+\alpha)+T_1}},
\end{align*}
where $\nu_1(y)=1+(1-\alpha)/y $ and $T_1$ was computed at
\eqref{equation_T1}. Furthermore, since $\mathcal{G}_\alpha>0$ almost surely,  we see by \eqref{e.3.7} 
that 
$\E{ \nu_1(\mathcal{G}_\alpha)}=p_{\mathcal{G}_\alpha}(0)=0$
(since $\alpha>1$).  By  the fact that $p_{\mathcal{G}_\alpha}(x)=0$ for $x\leq 0$,    we can write,
\begin{align*}
    0=p_{\mathcal{G}_\alpha}(x)=\E{\mathds{1}_{\{\mathcal{G}_{\alpha}> x\}}\nu_1(\mathcal{G}_\alpha)} =-\E{\mathds{1}_{\{\mathcal{G}_{\alpha}\leq x\}}\nu_1(\mathcal{G}_\alpha)}\,,\quad \forall \ x\le 0 .
\end{align*}
It follows from the two previous equations that for $x\leq 0$,  
\begin{align*}
    \abs{p_{F+\alpha}(x)-p_{\mathcal{G}_\alpha}(x)}=\abs{\E{\mathds{1}_{\{F+\alpha\leq x\}}\ \nu_1(F+\alpha)-\mathds{1}_{ \{ \mathcal{G}_\alpha \le x \}}
    \nu_1(\mathcal{G}_\alpha)  -    \mathds{1}_{\{F+\alpha \leq x\}}T_1} }\,.
\end{align*}
Deriving estimate for this quantity with $x\leq 0$ can be done in the same way as in the $x>0$ case, with the
difference that we use Proposition 
\ref{prop_steinmethod_x<0} and
Proposition \ref{prop_steinmethod_x=0} in place of Proposition \ref{prop_steinmethod_x>0}. Note that Proposition \ref{prop_steinmethod_x=0} is the reason we require $\alpha>1$ whenever $x=0$. 
\end{proof}
\begin{remark}
Note that the authors of \cite{HLN14} obtains
estimates for the difference between the densities of Gaussian
functionals and the normal density. In particular, their estimates are uniform with respect to $x\in \R$. It is hence natural to ask whether
we can do the same for the Gamma case. The
answer to this question relies on whether Proposition
\ref{prop_steinmethod_x>0} can be improved from a pointwise estimate
to a uniform estimate, i.e., whether one can control the supremum over
$x>0$ of $|f_{h_{k+1}}'(y)|$, where $h_{k+1}(y)=\mathds{1}_{\{y>x\}}\nu_{k+1}(y)$ and $\nu_{k+1}$
are defined at \eqref{def_nuk}. The arguments in the proof of
Proposition \ref{prop_steinmethod_x>0} seem to indicate that it is not
possible to take a supremum over $x>0$. Indeed, if $x\to\infty $, then at
\eqref{estimate_fh')<y<x}, the factor $e^x$ would diverge. Similarly,
in the case where $x\to 0$, factors like $x^{-1}$ or
$x^{\alpha-\ceil{\alpha}}$ appearing in \eqref{estimate_fh'y>x} would diverge. 
\end{remark}
With the above proposition at hand, we are now ready to prove the
first of our main results, namely Theorem \ref{theorem_fourthmoment_singlechaos_0thderivative}.
\begin{proof}[Proof of Theorem \ref{theorem_fourthmoment_singlechaos_0thderivative}]
Combining \cite[Theorem 3.4]{ACP14} which states
\begin{align*}
    \E{\brac{{q\brac{F+\alpha}} -w}^2}^{1/2}\leq \sqrt{\frac{q^2}{3}\brac{\E{F^4}-6\E{F^3}+6(1-\alpha)\alpha+3\alpha^2}}
\end{align*}
with Proposition \ref{prop_densityestimate_DF^2_0thderivative} yields
the desired conclusion.
\end{proof}
\subsection{Pointwise estimate for derivatives of densities}
~\

We proceed as in the previous subsection by stating and proving
preliminary results which will constitute the building blocks of the
proof of Theorem \ref{theorem_fourthmoment_singlechaos_kthderivative}.
We start with the following lemma which is a generalized version of
Lemma \ref{lemma_compute_D2F}.
\begin{lemma}
\label{lemma_DkDlbound}
Let $q \geq 2$ be an even integer and $f \in
\mathfrak{H}^{\otimes q}$. Let $k,l$ be integers such that $1\leq l\leq
k\leq q/2+1$ and $k+l\geq 3$. Define $F=I_q(f)$, then it holds that
\begin{align*}
      \E{\norm{\Lambda(k,l)}_{\mathfrak{H}^{\otimes k+l-2}}^2}&= \lambda(k,l)^2\sum_{\substack{r=0\\r\neq q/2-1}}^{q-k}\tau(k,l,r)^2(2q-k-l-2r)! \norm{f\widetilde{\otimes}_{r+1}f }^2_{\mathfrak{H}^{\otimes 2q-2-2r}}\\
    &\quad +\frac{q!^2}{(q-k-l+2)!}\norm{q(q/2-1)! {q-1\choose q/2-1}^2f\widetilde{\otimes}_{q/2} f-f }_{\mathfrak{H}^{\otimes q}}^2,
    \end{align*}
    where $\lambda(k,l)$ and $\Lambda(k,l)$ are the quantities defined
    at \eqref{def_lambdakl} and \eqref{def_Lambda}. Furthermore, the quantity $\tau(k,l,r)$
    is given by
\begin{align*}
    \tau(k,l,r)=\frac{q!^2}{(q-l)!(q-k)!}r!{q-k\choose r}{q-l\choose r}.
\end{align*}
Moreover, writing $\Theta = {q\brac{F+\alpha}} -\norm{DF}^2_\mathfrak{H}$, there exists a positive constant $C (q,k,l)$ dependent on $q,k$ and $l$ such that 
\begin{align}
\label{estimate_DkDlbymalliavindifference}
    \E{\norm{\Lambda(k,l)}_{\mathfrak{H}^{\otimes k+l-2}}^2}&\leq C (q,k,l)\E{\Theta^2}.
\end{align}
\end{lemma}

\begin{remark}
\label{remark_eveninteger_secondplace}
 Similar to Remark \ref{remark_eveninteger} about estimate \eqref{estimate_D^2F-qDF},
  we do not know if   \eqref{estimate_DkDlbymalliavindifference} is true for   any positive integer $q$, or  just 
  for an even $q$ greater or equal to $2. $
\end{remark}

\begin{proof}
As for any $m \geq 1$, $D^m F={q!}/{(q-m)!}I_{q-m}(f)$, the product formula for multiple
Wiener integrals \eqref{prodformula} implies that
\begin{align}
\label{formula_DkFDlF}
    D^kF\otimes_1D^lF&=\frac{q!^2}{(q-l)!(q-k)!}\sum_{r=0}^{q-k}r!{q-k\choose r}{q-l\choose r}I_{2q-k-l-2r}\brac{f\widetilde{\otimes}_{r+1}f}\nonumber\\
    &=\sum_{r=0}^{q-k}\tau(k,l,r) I_{2q-k-l-2r}\brac{f\widetilde{\otimes}_{r+1}f}.
\end{align}
By inductively applying the identity ${n-1\choose k }=\frac{n-k}{n}{n\choose k}$, one can deduce that 
\begin{align*}
    {q-k\choose q/2-1}=\frac{(q/2)!(q-k)!}{(q/2-k+1)!(q-1)!}{q-1\choose q/2-1}.
\end{align*}
Applying the above formula to $\tau(k,l,q/2-1)$ leads to
\begin{align}
\label{formula_tauklq/2-1}
   &\tau(k,l,q/2-1)\nonumber\\
   &=\frac{q!^2}{(q-l)!(q-k)!}(q/2-1)!{q-1\choose q/2-1}{q-1\choose q/2-1}\nonumber\\
   &\hspace{8em}\times\brac{\frac{(q/2)!(q-k)!}{(q/2-k+1)!(q-1)!}\frac{(q/2)!(q-l)!}{(q/2-l+1)!(q-1)!}}\nonumber\\
   &=\brac{\frac{q!}{(q-k-l+2)!}q (q/2-1)!{q-1\choose q/2-1}^2}\nonumber\\
   &\hspace{8em}\times q!\frac{(q-k-l+2)!}{(q-\ell)!(q-k)!}\frac{1}{q}\brac{\frac{(q/2)!(q-k)!}{(q/2-k+1)!(q-1)!}\frac{(q/2)!(q-l)!}{(q/2-l+1)!(q-1)!}}\nonumber\\
   &=\brac{\frac{q!}{(q-k-l+2)!}q (q/2-1)!{q-1\choose q/2-1}^2}\frac{1}{\lambda(k,l)},
\end{align}
where $\lambda(k,l)$ is defined by \eqref{def_lambdakl}. Therefore, \eqref{formula_DkFDlF} and \eqref{formula_tauklq/2-1} imply that
\begin{align*}
    \Lambda(k,l)&= \lambda(k,l)D^kF\otimes_1D^lF-D^{k+l-2}F\\
    &=\lambda(k,l)\sum_{\substack{r=0\\r\neq q/2-1}}^{q-k}\tau(k,l,r)I_{2q-k-l-2r}(f\widetilde{\otimes}_{r+1}f)\\
    &\quad +\frac{q!}{(q-k-l+2)!}I_{q-k-l+2}\brac{q (q/2-1)!{q-1\choose q/2-1}^2 f\widetilde{\otimes}_{q/2}f-f}.
\end{align*}
The orthogonality of Wiener chaos of different orders, that is $\E{I_p(f)I_q(g)}=0$ for $p\neq q$, along with the
isometry property of multiple Wiener integrals, namely $\E{I_q(f)^2}=q!\norm{f}_{\mathfrak{H}^{\otimes q}}^2$ for symmetric $f$,  yield
\begin{align*}
    \E{\norm{\Lambda(k,l)}_{\mathfrak{H}^{\otimes k+l-2}}^2}&= \lambda(k,l)^2\sum_{\substack{r=0\\r\neq q/2-1}}^{q-k}\tau(k,l,r)^2(2q-k-l-2r)! \norm{f\widetilde{\otimes}_{r+1}f }^2_{\mathfrak{H}^{\otimes 2q-2-2r}}\\
    &\quad +\frac{q!^2}{(q-k-l+2)!}\norm{q(q/2-1)! {q-1\choose q/2-1}^2f\widetilde{\otimes}_{q/2} f-f }_{\mathfrak{H}^{\otimes q}}^2.
\end{align*}
For $0\leq r\leq q-k$, the simple inequality $r!{q-k\choose r}{q-l\choose r}\leq (q-k)!(q-l)!$ implies that $\tau(k,l,r)\leq q!^2$. This yields
\begin{align*}
    \E{\norm{\Lambda(k,l)}^2_{\mathfrak{H}^{\otimes k+l-2}}}&\leq  \lambda(k,l)^2q!^4\sum_{\substack{r=0\\r\neq q/2-1}}^{q-k} \norm{f\widetilde{\otimes}_{r+1}f }^2_{\mathfrak{H}^{\otimes 2q-2-2r}}\\
    &\quad +\frac{q!^2}{(q-k-l+2)!}\norm{q(q/2-1)! {q-1\choose q/2-1}^2f\widetilde{\otimes}_{q/2} f-f }_{\mathfrak{H}^{\otimes q}}^2\,. 
\end{align*}
From  Lemma \ref{lemma_DF^2gammaconvergencecondition}
we see 
\[
\norm{f\widetilde{\otimes}_{r+1}f }^2_{\mathfrak{H}^{\otimes 2q-2-2r}}\le C_q\E{\Theta^2}\ 
 \hbox{and}\ 
 \norm{q(q/2-1)! {q-1\choose q/2-1}^2f\widetilde{\otimes}_{q/2} f-f }_{\mathfrak{H}^{\otimes q}}^2
 \le  C_q\E{\Theta^2}\,. 
 \]
 This 
yields the estimate \eqref{estimate_DkDlbymalliavindifference}.  
\end{proof}
The next proposition is an analogue of Proposition \ref{prop_densityestimate_DF^2_0thderivative} for the derivatives
of densities instead of the densities themselves.
\begin{proposition}
\label{prop_densityestimate_DF^2_kthderivative}
Let $F$ be a multiple Wiener integral of an even order $q\geq 2$ such
that $\E{F^2}=\alpha$ and $1/\norm{DF}^2_\frak{H}\in \cap_{p\geq 1}L^p$. Let $k\geq 1$ be an integer. Assume that the
quantities $\E{w^{-6i}}$, $\E{(F+\alpha)^{-6j}}$ and
$\E{(F+\alpha)^{2(\alpha-n)}}$, $1\leq i\leq 2k+2$, $1\leq j\leq k+1$,
$1\leq n\leq \ceil{\alpha}-1$, exist and are finite. Then, for every
$x\neq 0$ and $\alpha>0$, we have the pointwise density estimate 
\begin{align*}
     &\abs{p^{(k)}_{F+\alpha}(x)-p^{(k)}_{\mathcal{G}_\alpha}(x)}\leq P(x)\E{\brac{{q\brac{F+\alpha}} -w}^2}^{1/2},
\end{align*}
where $P(x)$ is a polynomial whose coefficients depend on $x$ and are finite for any choice of $x$. The variables of $P(x)$ are $\E{w^{-6i}}$, $\E{(F+\alpha)^{-6j}}$ and
$\E{(F+\alpha)^{2(\alpha-n)}}$, $1\leq i\leq 2k+2$, $1\leq j\leq k+1$,
$1\leq n\leq \ceil{\alpha}-1$. Moreover for $x=0$, the same estimate holds if $\alpha>k+1$.     
\end{proposition}
\begin{proof}
Let us start by considering the case where $x>0$. Per Corollary \ref{corollary_densitygamma} and Lemma \ref{lemma_decompositiondensityforgammaconvergence}, we have
\begin{align}
\label{equation_kthderivativeF+alpha_minuskthderivativeGalpha_x>0}
  p_{F+\alpha}^{(k)}(x)-p_{\mathcal{G}_\alpha}^{(k)}(x)&=\E{\mathds{1}_{\{F+\alpha>x\}}\nu_{k+1}(F+\alpha)}-\E{\mathds{1}_{\{\mathcal{G}_{\alpha}>x\}}\nu_{k+1}(\mathcal{G}_\alpha)}\nonumber\\
  &\quad +\E{\mathds{1}_{\{F+\alpha>x\}}T_{k+1}}.
\end{align}
Applying Proposition \ref{prop_steinmethod_x>0} allows us to bound the
term 
\begin{equation*}
\E{\mathds{1}_{\{F+\alpha>x\}}\nu_{k+1}(F+\alpha)}-\E{\mathds{1}_{\{\mathcal{G}_{\alpha}>x\}}\nu_{k+1}(\mathcal{G}_\alpha)}.
\end{equation*}
Next, the definition of the term $T_{k+1}$ given in Lemma \ref{lemma_decompositiondensityforgammaconvergence} implies that 
\begin{align}
\label{expand_indicatorT_k+1}
    &\abs{\E{\mathds{1}_{\{F+\alpha>x\}}T_{k+1}}}\leq \E{a^2}^{1/2}\E{\Theta^2}^{1/2}+\sum_{(l,m,n_1,\dots,n_p,q_1,\ldots,q_r)\in I}\nonumber\\
    &\qquad\quad\E{\prod_{i=1}^p \norm{D^{n_i}F}^2_{\mathfrak{H}^{\otimes n_i}}
      \prod_{i=1}^r\norm{ D^{q_i}F}^2_{\mathfrak{H}^{\otimes q_i}}
      \abs{b_{l,m,n_1,\dots,n_p,q_1,\ldots,q_r}}^2}^{1/2}\E{\norm{\Lambda(l,m)}_{\mathfrak{H}^{\otimes l+m-2}}^2}^{1/2},
\end{align}
where the index set $I$ is defined in Lemma
\ref{lemma_decompositiondensityforgammaconvergence} as well. Regarding
the terms on the right hand side of \eqref{expand_indicatorT_k+1}, Lemma \ref{lemma_DkDlbound} ensures that
\begin{align*}
    \E{\norm{\Lambda(l,m)}_{\mathfrak{H}^{\otimes l+m-2}}^2}^{1/2}\leq C_2(q,l,m)\E{\Theta^2}^{1/2}. 
\end{align*}
In particular, we use the fact that $q\geq 2$ is an even integer here.
 
 What remains to deal with is the term $\E{a^2}^{1/2}$ as well as the term
\begin{align*}
    \E{\prod_{i=1}^p \norm{D^{n_i}F}^2_{\mathfrak{H}^{\otimes n_i}} \prod_{i=1}^r\norm{ D^{q_i}F}^2_{\mathfrak{H}^{\otimes q_i}}\abs{b_{l,m,n_1,\dots,n_p,q_1,\ldots,q_r}}^2}^{1/2}. 
\end{align*}
We will only study the former quantity, as the latter can be bounded
in the exact same way. Lemma
\ref{lemma_decompositiondensityforgammaconvergence} states that $a$ is
a real polynomial and that every term in $a$ has the form
\begin{align*}
 \mathcal{T}= C\frac{1}{w^i}\frac{1}{(F+\alpha)^j}F^v\prod_{s\in I}\inner{D^{{l}}F\otimes_{1}D^{{m}}F \bigotimes_{i=1}^p D^{n_i}F,\bigotimes_{j=1}^r D^{q_j}F}_{\mathfrak{H}^{\otimes \sum_{i=1}^r q_i}}^{V_s},
\end{align*}
where $i,j,v$ and $\{V_s\colon s\in I\}$ are non-negative integers
such that $i\leq 2k+2$ and $j\leq k+1$, while $C$ denotes some real
constant. Furthermore, the hypercontractivity of Wiener chaos (Lemma \ref{lemma_hypercontract})
implies that for $m\leq q$ and any $p\geq 2$, 
\begin{align*}
    \E{\norm{D^mF}_{\mathfrak{H}^{\otimes m}}^p}^{1/p}\leq \tilde{C}\E{F^2}^{1/2}<\infty,
\end{align*}
where $\tilde{C}$ is some positive constant. Moreover, the following contraction
inequality holds (see for instance \cite[Lemma 2.4]{doblerpeccati_ustatcontraction}).
\begin{align*}
  \norm{  D^{l}F\otimes_{1}D^{m}F}_{\mathfrak{H}^{\otimes l+m-2}}\leq \norm{ D^{l}F}_{\mathfrak{H}^{\otimes l}}\norm{ D^{m}F}_{\mathfrak{H}^{\otimes m}}. 
\end{align*}
Combining the previous facts together with Minskowki's and H\"{o}lder's inequalities leads to 
\begin{align*}
    \E{a^2}^{1/2}&=\E{\brac{\sum \mathcal{T}}^2}^{1/2}\\
    &\leq \sum \E{{\mathcal{T}}^2}^{1/2}\\
    &\leq \sum \bar{C}\E{\frac{1}{w^{6i}}}^{1/6}\E{\frac{1}{(F+\alpha)^{6j}}}^{1/6},
\end{align*}
with $\bar{C}$ being some positive constant. Therefore,
$\E{a^2}^{1/2}$ is bounded by a real polynomial whose variables are
$\E{w^{-6i}}^{1/6}$ and $\E{(F+\alpha)^{-6j}}^{1/6}$ such that $i\leq
2k+2$ and $j\leq k+1$. Combining
\eqref{equation_kthderivativeF+alpha_minuskthderivativeGalpha_x>0},
Proposition \ref{prop_steinmethod_x>0} and
\eqref{expand_indicatorT_k+1} yields the desired conclusion in the case $x>0$. 

We now turn to the case where $x\leq 0$. The identity at \eqref{formula_densitymalcal} and Lemma \ref{lemma_decompositiondensityforgammaconvergence} imply that 
\begin{align*}
     p^{(k)}_{F+\alpha}(x)&=(-1)^k\E{\mathds{1}_{\{F+\alpha>x\}}
                            \delta\brac{G_k\frac{DF}{w}}}\\
  &=(-1)^k\E{\brac{1-\mathds{1}_{\{F+\alpha\leq x\}}}\delta\brac{G_k\frac{DF}{w}}}\\
    &=(-1)^{k+1}\E{\mathds{1}_{\{F+\alpha\leq x\}}\delta\brac{G_k\frac{DF}{w}}}\\
    &=\E{\mathds{1}_{\{F+\alpha\leq x\}} \brac{\nu_{k+1}(F+\alpha) +T_{k+1}}},
\end{align*}
where $\nu_{k+1}(y)$ is defined at \eqref{def_nuk}. Corollary
\ref{corollary_densitygamma} allows us to write that, whenever $x< 0$ and
$\alpha>0$, we have
\begin{align*}
    p^{(k)}_{\mathcal{G}_\alpha}(x)=\E{\mathds{1}_{\{\mathcal{G}_{\alpha}\leq x\}}\nu_{k+1}(\mathcal{G}_\alpha)}=0
\end{align*}
and whenever $\alpha>k+1$, 
\begin{align*}
    p^{(k)}_{\mathcal{G}_\alpha}(0)=\E{\mathds{1}_{\{\mathcal{G}_{\alpha}\leq 0\}}\nu_{k+1}(\mathcal{G}_\alpha)}=0. 
\end{align*}
In order to deal with bounding the quantity
\begin{align*}
    \abs{p^{(k)}_{F+\alpha}(x)-p^{(k)}_{\mathcal{G}_\alpha}(x)}=\abs{\E{\mathds{1}_{\{F+\alpha\leq x\}} \brac{\nu_{k+1}(F+\alpha) +T_{k+1}}}- \E{\mathds{1}_{\{\mathcal{G}_{\alpha}\leq x\}}\nu_{k+1}(\mathcal{G}_\alpha)}}
\end{align*}
when $x\leq 0$, we follow the same methodology as is the case where
$x>0$ which we addressed before, the only difference being that we
make use of Propositions \ref{prop_steinmethod_x<0} and
\ref{prop_steinmethod_x=0} instead of \ref{prop_steinmethod_x>0}. Note
that Proposition \ref{prop_steinmethod_x=0} is the reason we require $\alpha>k+1$ for $x=0$. 
\end{proof}
With the above proposition at hand, we are now ready to prove the
second  of our main results, namely Theorem \ref{theorem_fourthmoment_singlechaos_kthderivative}.
\begin{proof}[Proof of Theorem
  \ref{theorem_fourthmoment_singlechaos_kthderivative}]
Combining \cite[Theorem 3.4]{ACP14} which states
\begin{align*}
    \E{\brac{{q\brac{F+\alpha}} -w}^2}^{1/2}\leq \sqrt{\frac{q^2}{3}\brac{\E{F^4}-6\E{F^3}+6(1-\alpha)\alpha+3\alpha^2}}
\end{align*}
with Proposition \ref{prop_densityestimate_DF^2_kthderivative} yields
the desired conclusion.
\end{proof}

\section{General Gaussian functionals}
\label{section_generalrv}
In this section, we deal with the case where $F$ is a random variable
that has a possibly infinite Wiener chaos expansion. This is in
contrast with the previous section where we assumed that $F$ belongs
to a single Wiener chaos of an even order. It cannot be expected   to
obtain density estimates expressed in terms of the fourth moment of
$F$ in this case, however we can bound the pointwise distance between the  densities    in term of Malliavin derivatives. We will denote 
\begin{align*}
    \bar{w}= \inner{DF, -DL^{-1}F}_\mathfrak{H}\quad \hbox{for}\quad \E F  =0.
\end{align*}
We start with a technical lemma that will serve in the proof of the
density estimates mentioned above. 
\begin{lemma}
\label{lemma_consequenceofpoincare}
    Let $F\in\mathbb{D}^{2,s}$ with $s\geq 4$ such that $\E{F}=0$ and $\E{F^2}=\alpha$. Let $m$ be the largest even integer less than or equal to $s/2$. Then, for every $t\leq m\leq s/2$, it holds that
 \begin{align*}
        \E{\brac{\bar{w}-(F+\alpha)}^t}^{1/t}\leq \sqrt{m-1}\E{\norm{D\bar{w}-DF}^{s/2}_\mathfrak{H}}^{2/s}.
    \end{align*}
\end{lemma}
\begin{proof}
Using Lyapounov's inequality and the fact that
$\E{\bar{w}}=\E{F^2}=\alpha$ (this is a consequence of integration by
parts) so that $\E{\bar{w}-(F+\alpha)}=0$, we can write
 \begin{align*}
        \E{\brac{\bar{w}-(F+\alpha)}^t}^{1/t}\leq \E{\brac{\bar{w}-(F+\alpha)}^m}^{1/m}\leq \sqrt{m-1}\E{\norm{D\bar{w}-DF}^m_\frak{H}}^{1/m}.
 \end{align*}
In particular, the last inequality is due to the Poincar\'e  inequality in Lemma \ref{lemma_poincare}. Using Lyapounov's inequality again yields
  \begin{align*}
        \E{\brac{\bar{w}-(F+\alpha)}^t}^{1/t}\leq \sqrt{m-1}\E{\norm{D\bar{w}-DF}^{s/2}_\frak{H}}^{2/s}.
    \end{align*}
\end{proof}
We can now state the main result of this section, namely a density
estimate in terms of Malliavin operators. As before, we denote $\bar{w}= \inner{DF,
  -DL^{-1}F}_\mathfrak{H}$.
\begin{theorem}
\label{theorem_gammaconvergence_sumofchaos}
Let $F\in \mathbb{D}^{2,s}$, $\E{\abs{F}^{2p}}<\infty$ and $\E{\abs{\bar{w}}^{-r}}<\infty$ where $p>1,r>2,s\geq 4$ satisfy $1/p+2/r+3/s=1$.  Assume further that $\E{F}=0$,
$\E{F^2}=\alpha$ and that the quantities
$\E{(F+\alpha)^{-2}}$, $\E{\bar{w}^{-6}}$,
$\E{(F+\alpha)^{-2}\bar{w}^{-2}}$, $\E{\norm{DL^{-1}F}_\mathfrak{H}^3}$
and $\E{(F+\alpha)^{2(\alpha-i)}}$, $1\leq i\leq \ceil{\alpha}-1$, are finite. Then, for every $x\neq 0$ and $\alpha>0$, we have the pointwise density estimate 
\begin{align*}
  \abs{ p_{F+\alpha}(x)-p_{\mathcal{G}_\alpha}(x)} &\leq \bigg( d_1(x)\sum_{i=1}^{\ceil{\alpha}-1}\E{(F+\alpha)^{2(\alpha-i)}}^{1/2}+ {d_2(x)}\E{\frac{1}{(F+\alpha)^2}}^{1/2}+{d_3(x)}\\
 &\quad +\sqrt{\frac{s}{2}-1}\E{\frac{1}{\bar{w}^2}}^{1/2}+
   \sqrt{\frac{s}{2}-1}\E{\frac{(1-\alpha)^2}{\bar{w}^2(F+\alpha)^2}}^{1/2}\\
  &\quad +\E{\frac{1}{\bar{w}^6}}^{1/3}\E{\norm{DL^{-1}F}_\mathfrak{H}^3}^{1/3}\bigg)\E{\norm{D\bar{w}-DF}^{s/2}_\mathfrak{H}}^{2/s},
\end{align*}
where the factors $d_1(x)$, $d_2(x)$ and $d_3(x)$ are positive and
finite for every $x \neq 0$. Furthermore, for $x=0$, the same estimate holds if $\alpha>k+1$.  
\end{theorem}
\begin{proof}
Using the density representation formula given in \cite[Proposition
3.3]{HLN14}, we can write
    \begin{align*}
        p_{F+\alpha}(x)=\E{\mathds{1}_{\{F+\alpha>x\}}\delta\brac{\frac{-DL^{-1}F}{\bar{w}}} }.
    \end{align*}
As in the proof of Proposition
\ref{prop_densityestimate_DF^2_0thderivative}, we split the proof into
two parts: the case where $x>0$ and the case where $x\leq 0$. We will
treat the first case only, as the second one is similar and
simpler. In the case where $x>0$, the above density representation formula can be decomposed as
\begin{align*}
    &p_{F+\alpha}(x)=\E{\mathds{1}_{\{F+\alpha>x\}} \brac{\frac{F}{\bar{w}} +\frac{1}{\bar{w}^2}\inner{D\bar{w}, -DL^{-1}F}_\mathfrak{H}}}\\
    &=\E{ \mathds{1}_{\{F+\alpha>x\}}\brac{1+\frac{1-\alpha}{F+\alpha}} }+\E{\mathds{1}_{\{F+\alpha>x\}}\brac{\frac{F+\alpha}{\bar{w}}-1}}\\
    &\quad +\E{\mathds{1}_{\{F+\alpha>x\}} \frac{1}{\bar{w}^2}\inner{D\bar{w}-DF,-DL^{-1}F }_\mathfrak{H}}+\E{\mathds{1}_{\{F+\alpha>x\}}(1-\alpha)\brac{\frac{1}{\bar{w}}-\frac{1}{F+\alpha}}}\\
    &=\E{ \mathds{1}_{\{F+\alpha>x\}}\brac{1+\frac{1-\alpha}{F+\alpha}} }+\E{\mathds{1}_{\{F+\alpha>x\}}\frac{1}{\bar{w}}\brac{F+\alpha-\bar{w}}}\\
    &\quad +\E{\mathds{1}_{\{F+\alpha>x\}} \frac{1}{\bar{w}^2}\inner{D\bar{w}-DF,-DL^{-1}F }_\mathfrak{H}}+\E{\mathds{1}_{\{F+\alpha>x\}}\frac{1-\alpha}{\bar{w}(F+\alpha)}\brac{F+\alpha-\bar{w}}}. 
\end{align*}
Corollary \ref{corollary_densitygamma} states that the density
function of the Gamma random variable $\mathcal{G}_\alpha$ can be
represented as 
\begin{align*}
    p_{\mathcal{G}_\alpha}(x)= \E{\mathds{1}_{\{\mathcal{G}_{\alpha}>x\}}\brac{1+\frac{1-\alpha}{\mathcal{G}_\alpha}}}, 
\end{align*}
so that we get 
\begin{align*}
    p_{F+\alpha}(x)-p_{\mathcal{G}_\alpha}(x)&= \E{ \mathds{1}_{\{F+\alpha>x\}}\brac{1+\frac{1-\alpha}{F+\alpha}} }-\E{\mathds{1}_{\{\mathcal{G}_{\alpha}>x\}}\brac{1+\frac{1-\alpha}{\mathcal{G}_\alpha}}} \\
    &\quad +\E{\mathds{1}_{\{F+\alpha>x\}}\frac{1}{\bar{w}}\brac{F+\alpha-\bar{w}}}\\
    &\quad +\E{\mathds{1}_{\{F+\alpha>x\}}\frac{1-\alpha}{\bar{w}(F+\alpha)}\brac{F+\alpha-\bar{w}}}\\
    &\quad +\E{\mathds{1}_{\{F+\alpha>x\}} \frac{1}{\bar{w}^2}\inner{D\bar{w}-DF,-DL^{-1}F }_\mathfrak{H}}\\
    &=\mathcal{A}_1+\mathcal{A}_2+\mathcal{A}_3+\mathcal{A}_4. 
\end{align*}
The term $\mathcal{A}_1$ has been bounded in Proposition
\ref{prop_steinmethod_x>0} and we have
    \begin{align*}
      &\mathcal{A}_1 \leq \brac{  d_1(x)\sum_{i=1}^{\ceil{\alpha}-1}\E{(F+\alpha)^{2(\alpha-i)}}^{1/2}+ {d_2(x)}\E{\frac{1}{(F+\alpha)^2}}^{1/2}+{d_3(x)}}\\
      &\hspace{23em}\times\E{\brac{F+\alpha -\bar{w}}^2}^{1/2}.
    \end{align*}
Lemma \ref{lemma_consequenceofpoincare} gives us
\begin{align*}
    \E{\brac{F+\alpha -\bar{w}}^2}^{1/2}\leq \E{\norm{D\bar{w}-DF}^{s/2}_\mathfrak{H}}^{2/s}.
\end{align*}
Using H\"{o}lder's inequality and Lemma
\ref{lemma_consequenceofpoincare} again, we can write
\begin{align*}
    \mathcal{A}_2&\leq \sqrt{\frac{s}{2}-1}\E{\frac{1}{\bar{w}^2}}^{1/2}\E{\norm{D\bar{w}-DF}^{s/2}_\mathfrak{H}}^{2/s},\\
    \mathcal{A}_3&\leq \sqrt{\frac{s}{2}-1}\E{\frac{(1-\alpha)^2}{\brac{\bar{w}(F+\alpha)}^2}}^{1/2}\E{\norm{D\bar{w}-DF}^{s/2}_\mathfrak{H}}^{2/s}
\end{align*}
and  
\begin{align*}
        \mathcal{A}_4
        &\leq \E{\frac{1}{\bar{w}^6}}^{1/3}\E{\norm{DL^{-1}F}_\mathfrak{H}^3}^{1/3}\E{\norm{D\bar{w}-DF}_\mathfrak{H}^3}^{1/3}.
    \end{align*}
Gathering the previous estimates yields the desired density estimate
in the case where $x>0$.  
\end{proof}

\section{Application to random variables in the second Wiener chaos}
\label{sectionapplications}

In this section, we present an application of Theorem
\ref{theorem_fourthmoment_singlechaos_0thderivative}. We will consider random variable $F$ lives in the second Wiener chaos (a chaos known to
have Gamma distributed elements) and derive a density estimate between $F$ and a Gamma density.

It is a well-known fact (see for
instance \cite[Proposition 2.7.13]{nourdinpeccatibook}) that a random
variable $F$ in the second Wiener chaos can be represented as 
\begin{align*}
F=\sum_{i=0}^\infty \lambda_i\brac{I_1(e_i)^2-1},
\end{align*}
where the real numbers $\left\{\lambda_i\colon i \geq 0
  \right\}$ are such that $\lambda_{i}\geq {\lambda_{i+1}}>0$ for all $i\in\N$, and $\{e_i\}_{i\in\N}$ is an orthonormal basis of the Hilbert space $\mathfrak{H}$. Let us make the following assumption.
	\begin{customassumption}{A}~
		\label{assumA}
	$F$ is a random variable with the representation 
 \begin{align}
 \label{sum_assumptionA}
     F=\sum_{i=0}^\infty \zeta_i\brac{I_1(e_i)^2-1}
 \end{align}
such that
\begin{itemize}
\item[(1)] $F$ has finite variance,
\item[(2)] the real numbers $\left\{\zeta_i\colon i \geq 0
  \right\}$ are such that $\zeta_{i}\geq {\zeta_{i+1}}>0$ for all $i\in\N$,
\item[(3)] there exists a real number $\alpha>0$ such that $\alpha-\sum_{i=0}^\infty\zeta_i\geq 0$.
\end{itemize}
\end{customassumption}
The following lemma is a consequence of \cite[Lemma 7.1]{HLN14} on negative moments of $F$ and the norm of their Malliavin derivatives.
\begin{lemma}
\label{lemma_negativemoments_secondchaos}
Let $\alpha$ and $F$ be defined as above. Then, for any $\theta>0$, $\E{(F+\alpha)^{-2\theta}}<\infty$
if and only if there exists a natural number $N(\theta)>2\theta$ for which ${\zeta_{N(\theta)}}>0$. Similarly, for any $\eta>0$,  $\E{\norm{D(F+\alpha)}_{\mathfrak{H}}^{-2\eta}}<\infty$
if and only if there exists a natural number $N(\eta)>2\eta$ for which ${\zeta_{N(\eta)}}>0$. 
\end{lemma}
\begin{proof}
Since $\alpha-\sum_{i=0}^\infty\zeta_i\geq 0$, we have for any $\theta>0$,
\begin{align*}
\E{(F+\alpha)^{-2\theta}}&=\E{\brac{\sum_{i=0}^\infty \zeta_i I_1(e_i)^2+\alpha- \sum_{i=0}^\infty \zeta_i }^{-2\theta}}\\
&\leq \E{\brac{\sum_{i=0}^\infty \zeta_i I_1(e_i)^2 }^{-2\theta}}.
\end{align*}
Then, according to \cite[Lemma 7.1]{HLN14}, the right-hand side is
finite if and only if there exists a natural number
$N(\theta)>2\theta$ and $\zeta_{N(\theta)}>0$ in the representation
\eqref{sum_assumptionA}. If such an $N(\theta)$ exists, then there is some constant $C(\theta)$ such that
\begin{align*}
    \E{\brac{\sum_{i=0}^\infty \zeta_i I_1(e_i)^2 }^{-2\theta}} \leq C(\theta) N(\theta)^{-\theta}(\zeta_{N(\theta)})^{-2\theta}. 
\end{align*}
The proof of the finiteness of the negative moments of $\norm{D(F+\alpha)}_{\mathfrak{H}}^2$ is similar since $D(F+\alpha)=2\sum_{i=0}^\infty \zeta_i I_1(e_i)e_i$ and
\begin{align*}
    \norm{D(F+\alpha)}_{\mathfrak{H}}^2=4\sum_{i=0}^\infty \zeta_i^2 I_1(e_i)^2.
\end{align*}
Therefore, \cite[Lemma 7.1]{HLN14} is once more applicable here.  
\end{proof}
We are now ready to present our application. Informally, the next result say under suitable conditions, one can measure the density difference between a random variable $F$ in the second Wiener chaos and a Gamma random variable just by knowing the first four moments of $F$. 
\begin{theorem}
Let $F$ and $\alpha$ be the random variable and the parameter
satisfying Assumption \ref{assumA}, respectively. Let
$\mathcal{G}_{\alpha}$ denote a Gamma distributed random variable with density
\begin{align*}
    p_{\mathcal{G}_\alpha}(x)=\frac{1}{\Gamma(\alpha)}x^{\alpha-1}e^{-x}\mathds{1}_{(0,\infty)}(x).
\end{align*}
Assume further that there is some natural number $N>8$ such that
$\zeta_N>0$. Then, we have the density estimate
\begin{align*}
    \abs{p_{F+\alpha}(x)-p_{\mathcal{G}_\alpha}(x)}
    \leq C(x)\sqrt{\E{F^4}-6\E{F^3}+6(1-\alpha)\alpha+3\alpha^2},
\end{align*}
where $C(x)$ is a constant depending on $x$. 
\end{theorem}

\begin{remark}
\label{remark_tightestimate}
  The above estimate is tight. Suppose $F=\frac{1}{2}\sum_{i=1}^N \brac{I_1(e_i)^2-1}$ for some integer $N>8$ and $\mathcal{G}_\alpha=\mathcal{G}_{N/2}$, that is $\alpha=N/2$. Since $F$ is distributed as a chi-square distribution of $N$ degree of freedom, we can compute that  
    \begin{align*}
        \E{F^4}=\frac{3N^2}{4}+3N \quad\text{  and  }\quad  \E{F^3}=N. 
    \end{align*}
Then 
\begin{align*}
    \E{F^4}-6\E{F^3}+6(1-N/2)N/2+3(N/2)^2=0. 
\end{align*}
\end{remark}

\begin{proof}
Using Lemma \ref{lemma_negativemoments_secondchaos}, the fact that
there exists a natural number $N>8$ such that $\zeta_N>0$ implies that
$\E{(F+\alpha)^{-4}}<\infty$ and
$\E{\norm{DF}_\mathfrak{H}^{-8}}<\infty$. Moreover, as $F$ is an element
of the second Wiener chaos with finite variance, the
hypercontractivity property of the Wiener chaos per Lemma \ref{lemma_hypercontract} and Lyapunov's
inequality imply that $F$ must have finite $p$-th moments of all orders $p>0$. Consequently, 
\begin{align*}
    &\sum_{i=1}^{\ceil{\alpha}-1}\E{(F+\alpha)^{2(\alpha-i)}}^{1/2}+ \E{\frac{1}{(F+\alpha)^2}}^{1/2}
      +\E{\frac{1}{\norm{DF}_{\mathfrak{H}}^4}}^{1/2}\\
    &\qquad\qquad\qquad\qquad\qquad\qquad +\E{\frac{1}{\norm{DF}_{\mathfrak{H}}^4(F+\alpha)^2}}^{1/2} +\E{\frac{1}{\norm{DF}_{\mathfrak{H}}^6}}^{1/2}<\infty. 
\end{align*}
This allows us to apply Theorem
\ref{theorem_fourthmoment_singlechaos_0thderivative}, which yields the
desired density estimate.
\end{proof}

\bibliographystyle{amsplain}








\end{document}